\documentclass[12pt, reqno]{amsart}
\usepackage{amsmath, amsthm, amsfonts, amssymb, verbatim, graphicx}
\usepackage{a4wide}

\allowdisplaybreaks

\newcommand{\R}{\mathbb{R}}
\newcommand{\Q}{\mathbb{Q}}
\newcommand{\N}{\mathbb{N}}
\newcommand{\C}{\mathbb{C}}
\newcommand{\Z}{\mathbb{Z}}
\newcommand{\uhp}{\mathbb{H}}

\newcommand{\Log}{\textnormal{Log}}
\newcommand{\im}{\textnormal{Im}}
\newcommand{\re}{\textnormal{Re}}

\newtheorem{thm}{Theorem}
\newtheorem{cor}{Corollary}
\newtheorem{lem}{Lemma}
\newtheorem{prop}{Proposition}

\newtheorem{conjecture}{Conjecture}

\makeatletter
\newcommand{\pushright}[1]{\ifmeasuring@#1\else\omit\hfill$\displaystyle#1$\fi\ignorespaces}
\newcommand{\pushleft}[1]{\ifmeasuring@#1\else\omit$\displaystyle#1$\hfill\fi\ignorespaces}
\makeatother

{\theoremstyle{definition}
\newtheorem{claim}{Claim}
\newtheorem{defn}{Definition}
\newtheorem{rmk}{Remark}
}

\title[]{Differentiability of arithmetic Fourier series arising from Eisenstein series}
\author[]{Izabela Petrykiewicz}

\address[I. Petrykiewicz]{Universit\'e Joseph Fourier, Institut Fourier, 100 rue des maths, 38402 St Martin d'H\`{e}res, France; Max Planck Institute for Mathematics, Vivatsgasse 7, 53111 Bonn, Germany}
\email{petrykii@mpim-bonn.mpg.de}

\subjclass[2010]{Primary 26A24; Secondary 11A55, 11F03, 26A15, 26A27, 42A16}

\keywords{Eisenstein series, Differentiability, Modulus of continuity}

\date{\today}

\begin{document}

\begin{abstract}
Let $k\in \N^*$ be even. We consider two series
$ F_k(x)= \sum_{n=1}^\infty \frac{\sigma_{k-1}(n)}{n^{k+1}} \sin(2\pi n x)$  and $G_k(x)= \sum_{n=1}^\infty \frac{\sigma_{k-1}(n)}{n^{k+1}} \cos(2\pi n x),$ where $\sigma_{k-1}$ is the divisor function. 
They converge on $\R$ to continuous functions. In this paper, we examine the differentiability of $F_k$ and $G_k$. 
These functions are related to Eisenstein series and their (quasi-)modular properties allow us to apply the method proposed by Itatsu in 1981 in the study of the Riemann series. 
We focus on the case $k=2$ and we show that the sine series exhibits different behaviour with respect to differentiability than the cosine series. 
We prove that the differentiability of $F_2$ at an irrational $x$ is related to the fine diophantine properties of $x$. We estimate the modulus of continuity of $F_2$. We formulate a conjecture concerning differentiability of $F_k$ 
and $G_k$ for any $k$ even.
\end{abstract}

\maketitle

\section{Introduction and statement of the results}
 
In this paper, we study certain analytic properties of arithmetic Fourier series arising from Eisenstein series. 
Some of the results presented here have been already announced in \cite{P2}, where the sketches of the proofs were given. In this paper, we prove all these statements in detail and more.
 Let $k\in \N^*$ be even. The Eisenstein series of weight $k\geq 4$ is defined in the upper-half plane $\uhp$ by
$$ E_k(z)=\frac{1}{2\zeta(k)}\sum_{\substack{m,n \in \Z\\ (m,n)\neq (0,0)}} \frac{1}{(m+nz)^k}.$$
Its Fourier expansion is
\begin{equation}\label{fourier}
 E_k(z)=1-\frac{2k}{B_k}\sum_{n=1}^{\infty}\sigma_{k-1}(n)e^{2\pi i nz},
\end{equation}
where $B_k$ is the $k$-th Bernoulli number, and $\sigma_{k-1}(n)=\sum_{d|n} d^{k-1}.$ It is modular of weight $k$ under the action of $SL_2(\Z)$. If $k=2$, 
we consider
$$E_2(z)=\frac{3}{\pi^2}\lim_{\varepsilon \searrow 0}\left(\sum_{\substack{m,n \in \Z\\ (m,n)\neq (0,0)}} \frac{1}{(m+nz)^2|m+nz|^\varepsilon}\right)+\frac{3}{\pi\im(z)},$$
which defines a quasi-modular form under $SL_2(\Z)$ with the Fourier expansion as in Equation~(\ref{fourier}). 
The function $E_2$ can be viewed as a modular (or Eichler) integral on $SL_2(\Z)$ of weight $2$ with the rational period function $-\frac{2\pi i}{z}$, see for example \cite{Kn}. 

For each $k$ even, we consider the following two series
$$ F_k(x)= \sum_{n=1}^\infty \frac{\sigma_{k-1}(n)}{n^{k+1}} \sin(2\pi n x) \quad \text{ and } \quad
G_k(x)= \sum_{n=1}^\infty \frac{\sigma_{k-1}(n)}{n^{k+1}} \cos(2\pi n x).$$
Since $\sigma_{0}(n)=o(n^\varepsilon)$ for all $\varepsilon>0$ (see for example \cite[p. 83]{T}), these series converge on $\R$. We are interested in differentiability and modulus of continuity of $F_k$ and $G_k$. 
In particular, we focus on the case $k=2$. 
The functions are interesting, because they exhibit different behaviour concerning differentiability. 
For instance, we think that $G_2$ is differentiable at all irrational points, but due to technical difficulties we have not proved this. The differentiability of $F_2$ at $x$ depends on the continued fraction expansion of $x$. 
Already in 1933, Wilton in his work \cite{W} proved that there is a connection between some series involving divisor functions and continued fractions. In this paper (among other series) he considered the following two series
$$ \sum_{n=1}^\infty \frac{\sigma_{0}(n)}{n} \cos{2\pi n x}; \quad \sum_{n=1}^\infty \frac{\sigma_{0}(n)}{n} \sin{2\pi n x }.$$
He showed that the convergence of these series at $x$ depends on the diophantine properties of $x$.

\medskip

Our work is motivated by the example of the Riemann ``non-differentiable'' function which is defined as
\begin{equation*}
 S(x)=\sum_{n=1}^{\infty}\frac{1}{n^2}\sin(\pi n^2 x).
\end{equation*}
At the end of the 19th century, it was thought to be continuous but nowhere differentiable. 
Then in 1910s Hardy and Littlewood proved that $S(x)$ was indeed neither differentiable at any irrational point $x$, nor at rational points $x=\frac{p}{q}$ such that $p,q$ were not both odd, \cite{H, HL}. 
Later, in 1970 in \cite{G}, Gerver showed that $S(x)$ was in fact differentiable at rational points $\frac{p}{q}$ such that $p$ and $q$ are both odd, however his proof was long. 
In 1981, in a 4-page paper ``Differentiability of Riemann's Function" \cite{I}, Itatsu gave an alternative proof of differentiability of $S$ at these rational points. 
His method was based on the relationship between $S(x)$ and the theta function $\theta(z)=\sum_{n\in\Z}e^{i\pi n^2 z}$, which is an automorphic form of weight $\frac{1}{2}$ under the action of the $\theta$-modular group. 
He considered a complex-valued function $\mathcal{S}(x)=\sum_{n=1}^{\infty} \frac{1}{n^2 \pi i}e^{i n^2 \pi x},$ whose real part is $S(x)$. 
Then he obtained a functional equation for $\mathcal{S}$ from its relationship to $\theta$ and Jacobi identity satisfied by $\theta$, namely for all $0\neq\frac{p}{q}\in\Q$ we have:
$$\mathcal{S}\Big(\frac{p}{q}+h\Big) -\mathcal{S}\Big(\frac{p}{q}\Big)=R(p,q)p^{-1/2}e^{\pi i/4 \cdot h/|h|}|h|^{1/2}\frac{h}{|h|}-\frac{h}{2}+O(|h|^{3/2}),$$ 
where $R(p,q)$ is a constant that depends on $p$ and $q$ and is zero if and only if $p$ and $q$ are both odd. He read off the behaviour of $S$ around rational points from this equation. 
In 1991, Duistermaat used this method to study H\"{o}lder regularity exponent of $S(x)$ reproving the results on its differentiability on $\R$, see \cite{Du}. 

The approach developed by Itatsu has been implemented by various mathematicians. For example, Balazard and Martin used it in studying the differentiability of the function
$$A(x)=\int_{0}^{\infty}\{t\}\{xt\}\frac{1}{t^2}dt,$$
where $\{y\}$ is the fractional part of $y$. The function $A(x)$ is interesting, because the Riemann hypothesis can be reformulated in terms of $A(x)$ or more precisely, 
the Nyman and Beurling criterion can be rephrased in terms of $A(x)$, see \cite{BBLS}. It has been shown by B\'{a}ez-Duarte, Balazard, Landreau and Saias that for all $x>0$ we have 
$$
A(x)= \frac{1}{2}\log(x)+C+\frac{1}{2\pi^2 x}\sum_{n=1}^{\infty}\frac{\sigma_0(n)}{n^2}\cos(2\pi n x)
-\frac{x}{\pi^2}\int_{x}^{\infty} \sum_{n=1}^{\infty}\frac{\sigma_0(n)}{n^2}\cos(2\pi n t)dt,$$
where $C$ is a constant. In 2011, Balazard and Martin proved that $A$ is differentiable at $x$ if and only if $x>0\notin\Q$, 
and $\sum_{k =0}^\infty (-1)^k \frac{\log(q_{k+1}(x)}{q_k(x)}$ converges, where $q_i(x)$ is the denominator of the $i$th convergent of $x$, see \cite{BM1, BM2}. 

\medskip

We can now state our first theorem.
\begin{thm}\label{2rat}
 Neither $F_2$ nor $G_2$ is differentiable at any $x\in\Q$. However, $G_2$ is right and left differentiable at each $x\in\Q$.
\end{thm}
\textit{Cf.} \cite[Theorem 1.1]{P2}.
If we denote the left and the right derivative of $G_{2}$ at $\frac{p}{q}$ (where we take $p=0$, $q=1$ when $\frac{p}{q}=0$,) by $G'_{2}\left({\frac{p}{q}}^{-}\right), G'_{2}\left({\frac{p}{q}}^{+}\right)$ respectively, 
then $G'_{2}\left({\frac{p}{q}}^{-}\right) - G'_{2}\left({\frac{p}{q}}^{+}\right) = \frac{\pi^4}{3q^2}$.
\medskip

In order to state the results for irrational points, we need to fix the notation. 
We can write every number $x \in (0,1)$ as a continued fraction $x=\frac{1}{a_1(x)+\frac{1}{a_2(x)+\frac{1}{a_3(x)+ \frac{1}{...}}}}=[0;a_1(x), a_2(x), ..., a_i(x), ...],$ with $a_i(x) \in \N$ for all $i$. 
If $x\in \Q$, this representation is finite (i.e. there exist $k > 0$ such that $a_i(x)=0$ for all $i \geq k$), otherwise it is infinite. 
Let $T$ be the Gauss map, \textit{ie.} $T(0)=0$ and $T(x)=\frac{1}{x} \mod 1$ otherwise. For brevity, write $T^0(x)=x$ and  $T^k=T(T^{k-1}(x))$ if $k>0$. 
The partial quotients of $x$ can be calculated from the Gauss map by $a_i(x)= \left\lfloor \frac{1}{T^{i-1}x} \right\rfloor$, where $\left\lfloor y \right\rfloor$ is the floor function. 
Let $(\frac{p_n(x)}{q_n(x)})_n$ be the sequence of continued fraction approximations of $x$, that is $\frac{p_n(x)}{q_n(x)}=[0;a_1(x), a_2(x), ..., a_n(x)]$. 
The convergents can be obtained from partial quotients by the recurrence relations: 
${p_{n}(x)}={a_{n}(x)p_{n-1}(x)+p_{n-2}(x)}$, ${q_{n}(x)=a_{n}(x)q_{n-1}(x)+q_{n-2}(x)}$, for $n\geq0$, 
and $p_{-1}(x)=1$, $p_{-2}(x)=0$, $q_{-1}(x)=0$, $q_{-2}(x)=1$. We make the following definition.

\begin{defn}
 Let $x \in \R\setminus\Q$. We will say that $x$ is a square-Brjuno number if $$\sum_{n=0}^{\infty}\frac{\log q_{n+1}(x)}{q_n(x)^2}<\infty.$$ 
\end{defn}

In addition, we introduce two technical conditions:
\begin{equation}\label{xtilde}\tag{$\ast$}
  \lim_{n \to \infty} \frac{\log q_{n+4}(x)}{q_n(x)^2} =0;
\end{equation}
\begin{equation}\label{xtildere}\tag{$\ast\ast$}
  \lim_{n \to \infty} \frac{\log q_{n+3}(x)}{q_n(x)^2} =0, \text{ and } a_n(x)=1 \text{ for only finitely many } n.
\end{equation}
We note that square-Brjuno property and Conditions (\ref{xtilde}) and (\ref{xtildere}) are independent.

We have the following theorem.
\begin{thm}\label{2irrIm} 
\begin{enumerate}
 \item[(i)] If $x\in\R\setminus\Q$ is a square-Brjuno number satisfying (\ref{xtilde}) or (\ref{xtildere}), then $F_2$ is differentiable at $x$. 
 On the other hand, if $x\in\R\setminus\Q$ is not a square-Brjuno number, then $F_2$ is not differentiable at $x$.
 \item[(ii)] If $x\in\R\setminus\Q$ satisfies (\ref{xtilde}) or (\ref{xtildere}), then $G_2$ is differentiable at $x$.
\end{enumerate}
\end{thm}
\textit{Cf.} \cite[Theorem 1.3]{P2}.

We observe that Condition~(\ref{xtilde}) is satisfied for almost all $x$, whereas Condition~(\ref{xtildere}) holds for almost no $x$.
We believe that both conditions (\ref{xtilde}) and (\ref{xtildere}) could be removed in Theorem~\ref{2irrIm}, however the applied method does not enable us to do this, which we will demonstrate in Section~\ref{spirr}.
Moreover, almost all numbers are square-Brjuno. If $x\in\R\setminus\Q$ is not square-Brjuno, then it must be Liouville. 
It follows that the set of irrational numbers which are not square-Brjuno has both Lebesgue measure and Hausdorff dimension equal to 0.

\medskip

We are also interested in the modulus of continuity of $F_2$. We say that a real-valued function $f$ admits a modulus of continuity $g$, if for all $x,y$ in the domain of $f$ we have $|f(x)-f(y)|\leq g(|x-y|)$.
We say that a real-valued function $f$ admits a local modulus of continuity $g$ at a point $x$, if for all $y$ in the domain of $f$ we have
$|f(x)-f(y)|\leq g(|x-y|)$. We have the following result.

\begin{thm}\label{moc}
 For all $x \in (0,1)\setminus\Q$ and all $y \in (0,1)$, we have
\begin{equation}\label{emoc}
 |F_2(x)-F_2(y)| \leq C_1|x-y|\log\left(\frac{1}{|x-y|}\right)+C_2|x-y|,
\end{equation}
for some constants $C_1$, $C_2$ dependent only on $x$.
\end{thm}
If $x$ is square-Brjuno satisfying (\ref{xtilde}) or (\ref{xtildere}), then $C_1=0$. However, there exist $C_1>0, C_2$ absolute, such that (\ref{emoc}) is satisfied for all $x \in (0,1)\setminus\Q$ and all $y \in (0,1)$.

\medskip

We believe that we could extend our results to any even $k$. Therefore, we formulate the following conjecture.

\begin{conjecture}\label{gencase}
 Let $k\in\N^*$ be even. We have the following.
\begin{enumerate}
 \item[(i)] Neither $F_k$ nor $G_k$ is differentiable at any $x\in\Q$; however, $G_k$ is right and left differentiable at each $x\in\Q$.
 \item[(ii)] The function $G_k$ is differentiable at any $x\in\R\setminus\Q$.
 \item[(iii)] The function $F_k$ is differentiable at $x\in\R\setminus\Q$ if and only if 
   \begin{equation}\label{kbrjuno}
    \sum_{n=0}^{\infty}\frac{\log q_{n+1}(x)}{q_n(x)^k}<\infty.
   \end{equation}
\end{enumerate}
\end{conjecture}

In order to prove Conjecture~\ref{gencase} for $k\geq 4$, we would proceed as in the case $k=2$.
There are a lot of terms to analyse, but we believe that for any given $k\geq 4$ this method would work 
(adding a technical condition similar to (\ref{xtilde}) 
of the type $\frac{\log(q_{n+4})}{q_n^k}\to 0$). However the calculations become very long, and we do not do it explicitly. We present arguments justifying the conjecture.

\begin{rmk}
 In 1988, Yoccoz studied the function defined by $$B_1(x)=\sum_{n=0}^{\infty}xT(x)T^2(x)...T^{n-1}(x)\log\left(\frac{1}{T^n(x)}\right),$$  
now called Brjuno function, see \cite{Y, MMY1}. This series converges if and only if
$$\sum_{n=0}^{\infty} \frac{\log(q_{n+1}(x))}{q_n(x)} < \infty.$$ This condition is called Brjuno condition and was introduced by Brjuno in the study of certain problems in dynamical systems see \cite{B,Br2}. 
The points of convergence are called Brjuno numbers.
The Brjuno function satisfies a functional equation $B_1(x)=-\log(x)+xB_1\big(\frac{1}{x}\big)$ on $(0,1)$.
Marmi, Moussa and Yoccoz studied a generalised version of Brjuno function, namely they define a linear operator $T_{\alpha}f(x)=x^{\alpha}f\big(\frac{1}{x}\big)$ 
and then consider the equation $(1-T^{\alpha})B_f=f$ such that $B_f(x+1)=B_f(x)$, see \cite{MMY1, MMY2}. 
The ``$k$th-Brjuno condition'' in (\ref{kbrjuno}) corresponds to studying this equation with $\alpha=k$ and $f(x)=-\log(x)$.
\end{rmk}

Another method of analysing series of the type of $F_k$ and $G_k$ involves wavelet methods, and was proposed by Jaffard, in the study of the Riemann ``non-differentiable'' function, in \cite{J2}. 
Studying H\"{o}lder regularity exponent of $F_k$ and $G_k$ using this method enables to prove some cases of Conjecture~\ref{gencase}. 
For each $n$, we define $\kappa_n$ by the equality ${\big|x-\frac{p_n}{q_n}\big|= \frac{1}{q_n^{\kappa_n}}}$, and we let $\mu(x)=\limsup_{n\to\infty}\kappa_n, \nu(x)=\liminf_{n\to\infty}\kappa_n.$
It has been proved in \cite{P1} that for $k\geq4$ and $x\in\R\setminus\Q$, if $\frac{1}{\nu(x)}-\frac{1}{\mu(x)}< \frac{1}{k}$, then the H\"{o}lder regularity exponents of $F_k$ and $G_k$ at $x$ are both $1+\frac{k}{\mu(x)}$. 
If $\mu(x)<\infty$, then we conclude that both $F_k$ and $G_k$ are differentiable at $x$. 
The condition $\mu(x)<\infty$ implies (\ref{kbrjuno}), and we see that one direction of Conjecture~\ref{gencase} (iii) is true. 
It is also worth noting that for almost all $x$, $\mu(x)=\nu(x)=2$, and therefore the conjecture is proved for almost all $x$ for all $k\geq4$.

\medskip

The differentiability of $G_k$ could be also studied using the connection to $((y))$ map. 
Let $((y))= \{y\}-\frac{1}{2}$, its Fourier series is $((y))=-\frac{1}{\pi}\sum_{m=1}^{\infty}\frac{\sin(2\pi my)}{m}.$
Let $L_k(x)=2\pi^2 \sum_{r=1}^{\infty} \frac{((rx))}{r^k}.$
This function converges uniformly on $\R$, and it is integrable.
Then for $x\in\R$ we have
$$G_k(x)= \int_0^x L_k(t)dt+\zeta(2)\zeta(k+1).$$
We could then study the differentiability of $\int_0^x L_k(t)dt$. This approach was suggested to the author by Don Zagier. 

\medskip

The content of this paper was a part of the PhD thesis \cite{Pthese} conducted at Universit\'{e} Joseph Fourier in Grenoble and defended in September 2014.

\medskip

The paper is organised as follows. In Section~\ref{sprat} we prove Theorem~\ref{2rat}. Then, in Section~\ref{spirr} we prove Theorem~\ref{2irrIm}. In  Section~\ref{smoc} we prove Theorem~\ref{moc}. 
Finally, in Section~\ref{sgen} we give indications behind Conjecture~\ref{gencase}.

\section{Proof of Theorem~\ref{2rat}}\label{sprat}

\subsection{Preliminaries}

In order to prove Theorem~\ref{2rat} we will proceed as Itatsu in \cite{I}. Consider a complex valued function
$$\varphi_2(t)=\sum_{n=1}^{\infty}\frac{\sigma_1(n)}{n^3}e^{2\pi int},$$
whose imaginary part is $F_2$ and real part is $G_2$. For a matrix $\gamma=\bigl(\begin{smallmatrix}
a&b\\ c&d
\end{smallmatrix} \bigr) \in SL_2(\Z)$ and $z\in\C$, we will denote the fraction transformation by
$$\gamma\cdot z= \frac{az+b}{cz+d},$$
if $cz+d\in\C\setminus\{0\}$, and 
$ \gamma\cdot \left(-\frac{d}{c}\right)= \infty.$
For any $\gamma=\bigl(\begin{smallmatrix}
a&b\\ c&d
\end{smallmatrix} \bigr) \in SL_2(\Z)$ the Eisenstein series $E_2$ satisfies
\begin{equation}\label{eqmod2}
 E_2(z)=\frac{E_2(\gamma\cdot z)}{(cz+d)^2}-\frac{6}{i\pi}\frac{c}{(cz+d)}.
\end{equation}
Based on this equation, we will find a functional equation for $\varphi_2$.

\subsection{Functional equation for $\varphi_2$} 

We use the convention that $0\cdot\infty=0$ and throughout the paper we will work with the principal branch $-\pi <\arg(z) \leq \pi$ of $z\in\C$. We have the following proposition.
\begin{prop}\label{rnprop}
Let $\gamma=\bigl(\begin{smallmatrix}
a&b\\ c&d
\end{smallmatrix} \bigr) \in SL_2(\Z)$ with $c\neq 0$ and $x\in\R$. We have
\begin{multline*}
\varphi_2(x)=(cx+d)^4 \varphi_2(\gamma\cdot x)-\frac{i \pi^3}{3c^3}(cx+d)\Log(cx+d)+P_{-\frac{d}{c}}(x)\\
-\frac{\pi^2}{c^2}(cx+d)^2\Log(cx+d)+6\int_{-\frac{d}{c}}^{x}{c(ct+d)^2(c(x-t)-(ct+d))\varphi_2(\gamma\cdot t)dt},
\end{multline*}
where $\Log$ denotes the principal value of the complex logarithm and $P_{-\frac{d}{c}}(x) \in \C[x]$ is a polynomial of degree less than or equal to 3 that depends on $c$ and $d$.
\end{prop}
The proof of Proposition~\ref{rnprop} is very technical, therefore we will split the calculations into various lemmas and claims. Firstly, we note that $\varphi_2$ is differentiable in the upper-half plane, thus we have the following.

\begin{claim}\label{claim:diffphi}
 Let $z\in\uhp$. We have
\begin{align}
\varphi_2'(z) = & 2i\pi \sum_{n=1}^{\infty}{\frac{\sigma_1(n)}{n^2}e^{2i\pi nz}}, \label{fprime}\\
\varphi_2''(z)= & -4\pi^2\sum_{n=1}^{\infty}{\frac{\sigma_1(n)}{n}e^{2i\pi nz}}, \label{fdprime}\\
\varphi_2'''(z)= & -8i\pi^3\sum_{n=1}^{\infty}{\sigma_1(n)e^{2i\pi nz}} = \frac{i\pi^3}{3}E_2(z)-\frac{i\pi^3}{3} \label{ftprime}.
\end{align}
\end{claim}

We then find a functional equation for $\varphi_2''(z)$, which will be useful later.
\begin{lem}\label{fephi''}
  Let  $\gamma=\bigl(\begin{smallmatrix}
a&b\\ c&d
\end{smallmatrix} \bigr) \in SL_2(\Z)$ with $c\neq 0$ and $\tau, \alpha \in \uhp$. We have
\begin{multline*}
\varphi_2''(\tau)=\varphi_2''(\gamma\cdot\tau)-\varphi_2''(\gamma\cdot\alpha)-\frac{i\pi^3}{3c(c\tau+d)}+\frac{i\pi^3}{3c(c\alpha+d)}\\
-2\pi^2\Log(c\tau+d)+2\pi^2\Log(c\alpha+d)+\varphi_2''(\alpha)-\frac{i\pi^3}{3}\tau+\frac{i\pi^3}{3}\alpha.
\end{multline*}
\end{lem}
\begin{proof}
 Let  $\gamma=\bigl(\begin{smallmatrix}
a&b\\ c&d
\end{smallmatrix} \bigr) \in SL_2(\Z)$ with $c\neq 0$ and $\tau, \alpha \in \uhp$. We have
\begin{align*}
 \varphi_2''(\tau)=&-4\pi^2\left(\sum_{n=1}^\infty \frac{\sigma_1(n)}{n} e^{2i\pi n t}\right)=\frac{i\pi^3}{3}\int_{i \infty}^{\tau} (E_2(t)-1) dt  &\textnormal{ by } (\ref{fdprime})\\
= & \frac{i\pi^3}{3}\int_{i \infty}^{\alpha} (E_2(t)-1) dt+\frac{i\pi^3}{3}\int_{\alpha}^{\tau} (E_2(t)-1) dt\\
= &\varphi_2''(\alpha)+\frac{i\pi^3}{3}\int_{\alpha}^{\tau} \left(\frac{E_2(\gamma\cdot t)}{(ct+d)^2}-\frac{6}{i\pi}\frac{c}{(ct+d)}\right) dt-\frac{i\pi^3}{3}\tau+\frac{i\pi^3}{3}\alpha &\textnormal{ by } (\ref{eqmod2})\\
=&\varphi_2''(\alpha)+\varphi_2''(\gamma\cdot\tau)-\varphi_2''(\gamma\cdot\alpha)-\frac{i\pi^3}{3c(c\tau+d)}+\frac{i\pi^3}{3c(c\alpha+d)}\\
&-2\pi^2\Log(c\tau+d)+2\pi^2\Log(c\alpha+d)-\frac{i\pi^3}{3}\tau+\frac{i\pi^3}{3}\alpha,
\end{align*}
where $\Log$ denotes the principal value of the complex logarithm. 
\end{proof}

For $\gamma=\bigl(\begin{smallmatrix}
a&b\\ c&d
\end{smallmatrix} \bigr) \in SL_2(\Z)$ with $c\neq 0$ and $z \in \uhp$ define 
\begin{equation}\label{eq:fgamma}
f_\gamma(z)=\varphi_2''(z)-\varphi_2''(\gamma\cdot z)+\frac{i\pi^3}{3c(cz+d)}+2\pi^2\Log(cz+d)+\frac{i\pi^3}{3}z.
\end{equation}
The next claim shows that $f_\gamma$ depends only on $c$ and $d$.

\begin{claim}\label{claim:fgamma}
 For each $\gamma \in SL_2(\Z)$ the function $f_\gamma$ is constant on $\uhp$. Moreover, if $\gamma_1=\bigl(\begin{smallmatrix}
a_1&b_1\\ c&d
\end{smallmatrix} \bigr), \gamma_2=\bigl(\begin{smallmatrix}
a_2&b_2\\ c&d
\end{smallmatrix} \bigr) \in SL_2(\Z)$, then $f_{\gamma_1}=f_{\gamma_2}$. 
\end{claim}
\begin{proof}
 It follows from Lemma~\ref{fephi''} that $f_\gamma(\tau)=f_\gamma(\alpha)$ for all $\tau,\alpha\in\uhp$, hence it must be a constant function on $\uhp$. Let $f_\gamma(z)=f_\gamma$ for all $z\in \uhp$. For the second part, let 
 $\gamma_1=\bigl(\begin{smallmatrix}
a_1&b_1\\ c&d
\end{smallmatrix} \bigr), \gamma_2=\bigl(\begin{smallmatrix}
a_2&b_2\\ c&d
\end{smallmatrix} \bigr) \in SL_2(\Z)$.
Observe that the Lemma~\ref{fephi''} implies that
\begin{equation*}
 \varphi_2''(\gamma_1\cdot z)-\varphi_2''(\gamma_2\cdot z)=f_{\gamma_1}-f_{\gamma_2},
\end{equation*}
for all $z \in \uhp$. Since $f_{\gamma}$ does not depend on $z$, we have
\begin{equation*}
 \lim_{\substack{
   z \to -\frac{d}{c} \\
   \im(z)>0}} |\varphi_2''(\gamma_1\cdot z)-\varphi_2''(\gamma_2\cdot z)|=|f_{\gamma_1}-f_{\gamma_2}|.
\end{equation*}
Writing $z=x+iy$ we have
\begin{align}
 |\varphi_2''&(\gamma_1\cdot z)-\varphi_2''(\gamma_2\cdot z)|=4\pi^2\left|\sum_{n=1}^{\infty}{\frac{\sigma_1(n)}{n}e^{2i\pi n\gamma_1\cdot z}}-\sum_{n=1}^{\infty}{\frac{\sigma_1(n)}{n}e^{2i\pi n\gamma_2\cdot z}} \right| \notag\\
&= 4\pi^2\left| \sum_{n=1}^{\infty}{\frac{\sigma_1(n)}{n}\left(e^{2i\pi n \frac{a_1 x +b_1 + i a_1y}{cx+d+icy}}-e^{2i\pi n\frac{a_2 x +b_2+ia_1 y}{cx+d+icy}}\right)} \right|\notag\\
&\leq 4\pi^2 \sum_{n=1}^{\infty}{\frac{\sigma_1(n)}{n}e^{-\frac{2\pi n y}{(cx+d)^2+(cy)^2}}\left|e^{2i\pi n \frac{a_1c^2+b_1cx+a_1d x +b_1d+a_1cy}{(cx+d)^2+(cy)^2}}-e^{2i\pi n\frac{a_2c^2+b_2cx+a_2d x +b_2d+a_1cy}{(cx+d)^2+(cy)^2}}\right|}\notag\\
&\leq 8\pi^2 \sum_{n=1}^{\infty}{\frac{\sigma_1(n)}{n}e^{-\frac{2\pi n y}{(cx+d)^2+(cy)^2}}}.\label{snow}
\end{align} 
Since $x \to -\frac{d}{c}$ and $y \to 0^+$, as $z \to -\frac{d}{c}$, we conclude from (\ref{snow}) that $|\varphi_2''(\gamma_1\cdot z)-\varphi_2''(\gamma_2\cdot z)| \to 0$ as $z \to -\frac{d}{c}$. 
This shows that $f_{\gamma_1}=f_{\gamma_2}$. 
\end{proof}

We will now find a functional equation for $\varphi'_2$.

\begin{lem}\label{fephi'}
  Let  $\gamma=\bigl(\begin{smallmatrix}
a&b\\ c&d
\end{smallmatrix} \bigr) \in SL_2(\Z)$ with $c\neq 0$ and $\tau, \alpha \in \uhp$. We have
\begin{multline*}
\varphi_2'(\tau)=(c\tau+d)^2 \varphi_2'(\gamma\cdot \tau)-2c(c\tau+d)^3 \varphi_2(\gamma\cdot\tau)+6c^2\int_{\alpha}^{\tau}{(ct+d)^2 \varphi_2(\gamma\cdot t)dt}\\
-\frac{i\pi^3}{3c^2}\Log(c\tau+d)-2\pi^2\frac{(c\tau+d)}{c}\Log(c\tau+d)+Q_{\gamma, \alpha}(\tau),
\end{multline*}
where $Q_{\gamma, \alpha}(\tau) \in \C[\tau]$ of degree less than or equal to 2 depending on $\gamma$ and $\alpha$.
\end{lem}
\begin{proof}
Let  $\gamma=\bigl(\begin{smallmatrix}
a&b\\ c&d
\end{smallmatrix} \bigr) \in SL_2(\Z)$ with $c\neq 0$ and $\tau, \alpha \in \uhp$. We have
\begin{align}
\varphi_2'(\tau)
=&\frac{i\pi^3}{3}\int_{i \infty}^{\alpha}{(\tau-t)(E_2(t)-1)dt}+\frac{i\pi^3}{3}\int_{\alpha}^{\tau}{(\tau-t)E_2(t)dt}-\frac{i\pi^3}{3}\int_{\alpha}^{\tau}{(\tau-t)dt} &\textnormal{ by } (\ref{fprime})  \notag\\
=&(\tau-\alpha)\varphi_2''(\alpha)+\varphi_2'(\alpha)+\frac{i\pi^3}{3}\int_{\alpha}^{\tau}{(\tau-t)E_2(t)dt}-\frac{i\pi^3}{6}\tau^2+\frac{i\pi^3}{3}\alpha\tau-\frac{i\pi^3}{6}\alpha^2. \label{eq1}
\end{align}
We apply the relationship (\ref{eqmod2}) and we integrate the remaining integral:
\begin{align*}
 \int_{\alpha}^{\tau}&{(\tau-t)E_2(t)dt}=\int_{\alpha}^{\tau}{(\tau-t)\Big(\frac{1}{(ct+d)^2}E_2(\gamma\cdot t)-\frac{6c}{i\pi (ct+d)}\Big)dt}\\
=&\int_{\alpha}^{\tau}{\frac{(\tau-t)}{(ct+d)^2}E_2(\gamma\cdot t)dt}-\frac{6}{i\pi}\Big(\frac{(c\tau+d)}{c}\Log(c\tau+d)-\frac{(c\tau+d)}{c}\Log(c\alpha+d)-\tau+\alpha\Big)\\
=&\frac{3}{i\pi^3}\int_{\alpha}^{\tau}{\frac{(\tau-t)}{(ct+d)^2}\varphi_2'''(\gamma\cdot t)dt}-\frac{1}{c^2}\Log(c\tau+d)+\frac{1}{c^2}\Log(c\alpha+d)+\frac{(c\tau+d)}{c^2(c\alpha+d)}-\frac{1}{c^2}\\
&-\frac{6}{i\pi}\Big(\frac{(c\tau+d)}{c}\Log(c\tau+d)-\frac{(c\tau+d)}{c}\Log(c\alpha+d)-\tau+\alpha\Big) \quad \quad \quad \quad \quad \quad \quad\textnormal{ by } (\ref{fprime})\\
=&\frac{3}{i\pi^3}\Big(-(\tau-\alpha)\varphi_2''(\gamma\cdot \alpha)+(c\tau+d)^2 \varphi_2'(\gamma\cdot\tau)-(c\alpha+d)^2 \varphi_2'(\gamma\cdot \alpha)\\
&-2c(c\tau+d)^3 \varphi_2(\gamma\cdot\tau)+3c(c\alpha+d)^3 \varphi_2(\gamma\cdot\alpha)+6c^2\int_{\alpha}^{\tau}{(ct+d)^2 \varphi_2(\gamma\cdot t)dt}\Big)\\
&-\frac{1}{c^2}\Log(c\tau+d)+\frac{1}{c^2}\Log(c\alpha+d)+\frac{(c\tau+d)}{c^2(c\alpha+d)}-\frac{1}{c^2}\\
& -\frac{6}{i\pi}\Big(\frac{(c\tau+d)}{c}\Log(c\tau+d)-\frac{(c\tau+d)}{c}\Log(c\alpha+d)-\tau+\alpha\Big).
\end{align*}
Substituting it into (\ref{eq1}) gives
\begin{multline*}
\varphi_2'(\tau)=(c\tau+d)^2 \varphi_2'(\gamma\cdot\tau)-2c(c\tau+d)^3 \varphi_2(\gamma\cdot\tau)+6c^2\int_{\alpha}^{\tau}{(ct+d)^2 \varphi_2(\gamma\cdot t)dt}\\
-\frac{i\pi^3}{3c^2}\Log(c\tau+d)-2\pi^2\frac{(c\tau+d)}{c}\Log(c\tau+d)+Q_{\gamma, \alpha}(\tau),
\end{multline*}
where
$Q_{\gamma, \alpha}(\tau)=B'\tau^2+C'\tau+D',$
with
\begin{align*}
 B'=& -\frac{i\pi^3}{6}\\
 C'=& \varphi_2''(\alpha)-\varphi_2''(\gamma\cdot\alpha)+\frac{i\pi^3}{3c(c\alpha+d)}+2\pi^2\Log(c\alpha+d)+2\pi^2+\frac{i\pi^3}{3}\alpha\\
   =&  f_{\gamma}+2\pi^2  & \textnormal{ by Lemma } \ref{fephi''}
\end{align*}
\begin{align*}
 D'=& -\alpha (\varphi_2''(\alpha)-\varphi_2''(\gamma\cdot\alpha))+\varphi_2'(\alpha)-(c\alpha+d)^2 \varphi_2'(\gamma\cdot\alpha)+2c(c\alpha+d)^3 \varphi_2(\gamma\cdot\alpha)\\
    & +\frac{i\pi^3}{3c^2}\Log(c\alpha+d)+2\pi^2\frac{d}{c}\Log(c\alpha+d)+\frac{i\pi^3}{3}\frac{d}{c^2(c\alpha+d)}\\
    & -\frac{i\pi^3}{3c^2}-2\pi^2\alpha-\frac{i\pi^3}{6}\alpha^2\\
   =& -\alpha f_\gamma +\varphi_2'(\alpha)-(c\alpha+d)^2 \varphi_2'(\gamma\cdot\alpha)+2c(c\alpha+d)^3 \varphi_2(\gamma\cdot\alpha)\\
    & +\frac{i\pi^3}{3c^2}\Log(c\alpha+d)+(c\alpha+d)\frac{2\pi^2}{c}\Log(c\alpha+d)-2\pi^2\alpha+\frac{i\pi^3}{6}\alpha^2 \quad  \textnormal{ by Lemma } \ref{fephi''}.
\end{align*}
This completes the proof of the lemma.
\end{proof}

For $\gamma=\bigl(\begin{smallmatrix}
a&b\\ c&d
\end{smallmatrix} \bigr) \in SL_2(\Z)$ with $c\neq 0$ and $\rho, z \in \uhp$ define 
$g_\gamma(z,\rho)
=-z f_\gamma +\varphi_2'(z)-(cz+d)^2 \varphi_2'(\gamma\cdot z)+2c(ct+d)^3 \varphi_2(\gamma\cdot z) 
+\frac{i\pi^3}{3c^2}\Log(cz+d) +(cz+d)\frac{2\pi^2}{c}\Log(cz+d) -2\pi^2z+\frac{i\pi^3}{6}z^2-6c^2\int_{\rho}^{z}{(ct+d)^2 \varphi_2(\gamma\cdot t)dt}$. 
The next claim shows that $g_\gamma$ depends only on $\rho$ and $\gamma$.

\begin{claim}
  For each $\gamma \in SL_2(\Z)$, for all $\rho \in \uhp$ we have $g_\gamma(z,\rho)=g_\gamma(w,\rho)$ for all $z,w \in\uhp$. 
\end{claim}
\begin{proof}
 It follows from Lemma~\ref{fephi'}.
\end{proof}
For all $z\in\uhp$ write $g_\gamma(z,\rho)=g_\gamma(\rho)$. We note that Lemma~\ref{fephi'} implies that
\begin{multline}\label{eq2}
\varphi_2'(\alpha)-(c\alpha+d)^2 \varphi_2'(\gamma\cdot\alpha)=g_{\gamma}(\alpha)+\alpha f_\gamma-2c(c\alpha+d)^3 \varphi_2(\gamma\cdot\alpha) -\frac{i\pi^3}{3c^2}\Log(c\alpha+d) \\
-(c\alpha+d)\frac{2\pi^2}{c}\Log(c\alpha+d) +2\pi^2\alpha-\frac{i\pi^3}{6}\alpha^2.
\end{multline}

\medskip

We can now prove Proposition~\ref{rnprop}.

\noindent\textit{Proof of Proposition~\ref{rnprop}.} Fix $\alpha\in\uhp$ and $\gamma=\bigl(\begin{smallmatrix}
a&b\\ c&d
\end{smallmatrix} \bigr) \in SL_2(\Z)$ with $c\neq 0$, let $\tau \in\uhp$. Integrating by parts we get
\begin{align}
 \varphi_2(\tau)=&\frac{i\pi^3}{6} \int_{i \infty}^{\tau} (\tau-t)^2 (E_2(t)-1) dt \notag\\
=& \frac{i\pi^3}{6} \int_{\alpha}^{\tau}{(\tau-t)^2E_2(t)dt}-\frac{i\pi^3}{6} \int_{\alpha}^{\tau}{(\tau-t)^2}dt+\frac{i\pi^3}{6} \int_{i\infty}^{\alpha}{(\tau-t)^2(E_2(t)-1)dt}\notag\\
=& \frac{i\pi^3}{6} \int_{\alpha}^{\tau}{(\tau-t)^2E_2(t)dt}+\frac{i\pi^3(\alpha-\tau)^3}{18}+\frac{(\tau-\alpha)^2}{2}\varphi_2''(\alpha)+(\tau-\alpha)\varphi_2'(\alpha)+\varphi_2(\alpha).\label{eq3}
\end{align}
We apply (\ref{eqmod2}) to the first term, we then obtain:
\begin{align}
\frac{i\pi^3}{6}& \int_{\alpha}^{\tau}{(\tau-t)^2E_2(t)dt}= \frac{i\pi^3}{6}\int_{\alpha}^{\tau}{(\tau-t)^2\left(\frac{1}{(ct+d)^2}E_2(\gamma\cdot t)-\frac{6c}{i\pi (ct+d)}\right)dt}\notag\\
=&\frac{i\pi^3}{6}\int_{\alpha}^{\tau}{\frac{(\tau-t)^2}{(ct+d)^2}E_2(\gamma\cdot t)dt}-\frac{\pi^2}{c^2}\Big((c\tau+d)^2 \Log(c\tau+d)-(c\tau+d)^2\Log(c\alpha+d) \notag\\
&-2(c\tau+d)^2+2(c\tau+d)(c\alpha+d)+\frac{(c\tau+d)^2}{2}-\frac{(c\alpha+d)^2}{2}\Big).\label{eq4}
\end{align}
By (\ref{ftprime}), using the substitution $u=\gamma\cdot t$ and integrating by parts we get:
\begin{align}
 \frac{i\pi^3}{6}&\int_{\alpha}^{\tau}\frac{(\tau-t)^2}{(ct+d)^2}E_2(\gamma\cdot t)dt
=\frac{1}{2}\int_{\alpha}^{\tau}{\frac{(\tau-t)^2}{(ct+d)^2}\varphi_2'''(\gamma\cdot t)dt}+\frac{i\pi^3}{6}\int_{\alpha}^{\tau}{\frac{(\tau-t)^2}{(ct+d)^2}dt}\notag\\
=&\int_{\alpha}^{\tau}{(\tau-t)\varphi_2''(\gamma\cdot t)dt}-\frac{1}{2}(\tau-\alpha)^2 \varphi_2''(\gamma\cdot \alpha)-\frac{i\pi^3}{3c^3}(c\tau+d)\Log(c\tau+d)\notag\\
&+\frac{i\pi^3}{6c^3}\Big(\frac{(c\tau+d)^2}{(c\alpha+d)}+2(c\tau+d)\Log(c\alpha+d)-(c\alpha+d)\Big)\notag\\
=&-\int_{\alpha}^{\tau}{(ct+d)(2c\tau-3ct-d)\varphi_2'(\gamma\cdot t)dt}-(c\alpha+d)^2(\tau-\alpha)\varphi_2'(\gamma\cdot\alpha)\notag\\
&-\frac{(\tau-\alpha)^2}{2} \varphi_2''(\gamma\cdot \alpha)-\frac{i\pi^3}{3c^3}(c\tau+d)\Log(c\tau+d)\notag\\
&+\frac{i\pi^3}{6c^3}\Big(\frac{(c\tau+d)^2}{(c\alpha+d)}+2(c\tau+d)\Log(c\alpha+d)-(c\alpha+d)\Big)\notag\\
=&(c\tau+d)^4 \varphi_2(\gamma\cdot\tau) +2(c\tau+d)(c\alpha+d)^3\varphi_2(\gamma\cdot\alpha)-3(c\alpha+d)^4\varphi_2(\gamma\cdot\alpha)\notag\\
&+\int_{\alpha}^{\tau}{c(ct+d)^2(6c(\tau-t)-6(ct+d))\varphi_2(\gamma\cdot t)dt}-(c\alpha+d)^2(\tau-\alpha)\varphi_2'(\gamma\cdot\alpha)\notag\\
&-\frac{(\tau-\alpha)^2}{2} \varphi_2''(\gamma\cdot \alpha)-\frac{i\pi^3}{3c^3}(c\tau+d)\Log(c\tau+d)\notag\\
&+\frac{i\pi^3}{6c^3}\Big(\frac{(c\tau+d)^2}{(c\alpha+d)}+2(c\tau+d)\Log(c\alpha+d)-(c\alpha+d)\Big). \label{eq5}
\end{align}
Substituting (\ref{eq4}) and (\ref{eq5}) into (\ref{eq3}) and gathering the terms we get
\begin{multline*}
\varphi_2(\tau)=(c\tau+d)^4 \varphi_2(\gamma\cdot \tau)-\frac{i \pi^3}{3c^3}(c\tau+d)\Log(c\tau+d)+P_{\alpha,\gamma}(\tau)\\
-\frac{\pi^2}{c^2}(c\tau+d)^2\Log(c\tau+d)+6\int_{\alpha}^{\tau}{c(ct+d)^2(c(\tau-t)-(ct+d))\varphi_2(\gamma\cdot t)dt},
\end{multline*}
with $P_{\alpha,\gamma}(\tau)=A\tau^3+B\tau^2+C\tau+D$, where
\begin{align*}
 A=& -\frac{i\pi^3}{18}\\
 B=& \frac{1}{2}(\varphi_2''(\alpha)-\varphi_2''(\gamma\cdot \alpha))+\frac{i\pi^3}{6}\alpha+\pi^2\Log(c\alpha+d)+\frac{3\pi^2}{2} +\frac{i\pi^3}{6c(c\alpha+d)}\\
=& \frac{1}{2}f_\gamma+\frac{3\pi^2}{2} &\textnormal{ by Lemma }\ref{fephi''}
\end{align*}
\begin{align*}
 C=& -\alpha(\varphi_2''(\alpha)-\varphi_2''(\gamma\cdot\alpha))+\varphi_2'(\alpha)-(c\alpha+d)^2\varphi_2'(\gamma\cdot\alpha)+2c(c\alpha+d)^3 \varphi_2(\gamma\cdot\alpha)+\frac{3\pi^2d}{c} \\
   & +\frac{2\pi^2 d}{c}\Log(c\alpha+d)+\frac{i\pi^3}{3c^2}\Log(c\alpha+d) -\frac{i\pi^3}{6}\alpha^2-\frac{2\pi^2}{c}(c\alpha+d)+ \frac{i\pi^3 d}{3c^2(c\alpha+d)}\\
  =&H_{\gamma}(\alpha) +\frac{i\pi^3}{3c^2}+\frac{\pi^2 d}{c}
    \quad \quad \quad \quad \quad \quad \quad \quad \quad \quad \quad \quad \quad \quad \quad \quad \quad \textnormal{ by Lemma }\ref{fephi''}\textnormal{ and }(\ref{eq2}) \\
 D=& \frac{1}{2}\alpha^2(\varphi_2''(\alpha)-\varphi_2''(\gamma\cdot\alpha))-\alpha(\varphi_2'(\alpha)-(c\alpha+d)^2\varphi_2'(\gamma\cdot\alpha))
    -(c\alpha+d)^4\varphi_2(\gamma\cdot\alpha)\\
   & -2c(c\alpha+d)^3 \alpha \varphi_2(\gamma\cdot\alpha))+\varphi_2(\alpha)
   +\frac{\pi^2 d^2}{c^2}\Log(c\alpha+d)+\frac{i\pi^3 d}{3c^3}\Log(c\alpha+d)\\
   &+\frac{i\pi^3}{18}\alpha^3+\frac{3\pi^2 d^2}{2c^2}-\frac{2\pi^2 d}{c^2}(c\alpha+d)+\frac{\pi^2}{2c^2}(c\alpha + d)^2
   +\frac{i\pi^3 d^2}{6c^3(c\alpha+d)}-\frac{i\pi^3}{6c^3}(c\alpha+d)\\
   =& -\frac{1}{2}\alpha^2 f_\gamma-\alpha g_{\gamma}(\alpha)-(c\alpha+d)^4\varphi_2(\gamma\cdot\alpha) +\varphi_2(\alpha)+(c\alpha+d)\frac{i\pi^3}{3c^3}\Log(c\alpha+d) \\
   &+(c\alpha +d)^2\frac{\pi^2}{c^2}\Log(c\alpha+d) -2\pi^2\alpha^2 +\frac{i\pi^3}{18}\alpha^3+\frac{3\pi^2 d^2}{2c^2}-\frac{2\pi^2 d}{c^2}(c\alpha+d)\\
   &+\frac{\pi^2}{2c^2}(c\alpha + d)^2 -\alpha\frac{i\pi^3}{3c^2}. 
    \quad \quad \quad \quad \quad \quad \quad \quad \quad \quad \quad \quad \quad \quad \quad \quad \textnormal{ by Lemma }\ref{fephi''}\textnormal{ and }(\ref{eq2})
\end{align*}
Then we observe that if we let $\alpha \to -\frac{d}{c}$, then $A,B,C,D$ are well defined. Moreover, since $D=\varphi_2(-\frac{d}{c})$ we have 
$g_{\gamma}\left(-\frac{d}{c}\right)=\frac{\pi^2 d}{2c}+\frac{i\pi^3 d^2}{18c^2}-\frac{i\pi^3}{3c^2}+\frac{d}{2c}f_\gamma.$ Therefore, we obtain
 $A= -\frac{i\pi^3}{18}, B= \frac{1}{2}f_\gamma+\frac{3\pi^2}{2},C= \frac{d}{2c}f_\gamma+\frac{3\pi^2 d}{2c}+\frac{i\pi^3 d^2}{18c^2},
 D= \varphi_2\left(-\frac{d}{c}\right)$. By Claim~\ref{claim:fgamma}, we deduce that the polynomial $P_{-\frac{d}{c},\gamma}$ depends only on $d$ and $c$. Write $P_{-\frac{d}{c},\gamma}=P_{-\frac{d}{c}}$. Hence we have
\begin{multline*}
\varphi_2(\tau)=(c\tau+d)^4 \varphi_2(\gamma\cdot \tau)-\frac{i \pi^3}{3c^3}(c\tau+d)\Log(c\tau+d)+P_{-\frac{d}{c}}(\tau)\\
-\frac{\pi^2}{c^2}(c\tau+d)^2\Log(c\tau+d)+6\int_{-\frac{d}{c}}^{\tau}{c(ct+d)^2(c(\tau-t)-(ct+d))\varphi_2(\gamma\cdot t)dt}.
\end{multline*}
Letting $\tau \to x\in\R$ gives the result. \hfill\qedsymbol

\subsection{Proof of Theorem~\ref{2rat}} 

Before we start proving Theorem~\ref{2rat}, we rewrite the polynomial $P_{-\frac{d}{c}}$ as $P_{-\frac{d}{c}}(x)=\tilde{A}(cx+d)^3+\tilde{B}(cx+d)^2+\tilde{C}(cx+d)+\tilde{D}$ with
$$ \tilde{A}= -\frac{ i\pi^3}{18c^3}; \;
 \tilde{B}=\frac{f_{\gamma}}{2c^2}+\frac{3\pi^2}{2c^2}+\frac{i\pi^3d}{6c^3}; \;
 \tilde{C}= -\frac{d}{2c^2}f_{\gamma}+\frac{3\pi^2 d}{2c^2}-\frac{d^2 i\pi^3}{9c^3}; \;
 \tilde{D}= \varphi_2\left(-\frac{d}{c}\right).$$

\medskip

\noindent\textit{Proof of Theorem~\ref{2rat}.} Let $\frac{p}{q}\in\Q$, $p,q$ coprime, if $x=0$, then let $q=1$, $p=0$. By B\'{e}zout's identity, we can choose $\gamma=\bigl(\begin{smallmatrix}
a&b\\ q&-p
\end{smallmatrix} \bigr) \in SL_2(\Z)$. By Proposition~\ref{rnprop} we have
\begin{multline*}
\varphi_2(x)=(qx-p)^4 \varphi_2(\gamma\cdot x)-\frac{i \pi^3}{3q^3}(qx-p)\Log(qx-p)+P_{\frac{p}{q}}(x)\\
-\frac{\pi^2}{q^2}(qx-p)^2\Log(qx-p)+6\int_{\frac{p}{q}}^{x}{q(qt-p)^2(q(x-t)-(qt-p))\varphi_2(\gamma\cdot t)dt}.
\end{multline*}
We observe that since $\varphi_2(x)$ is bounded on $\R$ we have
\begin{multline*}
 \left|6\int_{\frac{p}{q}}^{x}{q(qt-p)^2(q(x-t)-(qt-p))\varphi_2(\gamma\cdot t)dt}\right|\\
\leq c_1 \left|\int_{\frac{p}{q}}^{x}{(qt-p)^2(q(x-t)-(qt-p))dt}\right|
\leq c_2 (qx-p)^4,
\end{multline*}
for some constants $c_1, c_2$.

\medskip

As $x\to \frac{p}{q}^+$, $\Log$ becomes the natural logarithm $\log$, and we have
\begin{equation}\label{eq:rn+}
\varphi_2(x)=\varphi_2\Big(\frac{p}{q}\Big)-\frac{i \pi^3}{3q^3}(qx-p)\log(qx-p)+\tilde{C}(qx-p) +O((qx-p)^2\log(qx-p)). 
\end{equation}
Taking the imaginary part of the both sides of Equation~(\ref{eq:rn+}) shows that $F_2$ is not differentiable at $\frac{p}{q}$. 
On the other hand, taking the real part of the both sides of Equation~(\ref{eq:rn+}) shows that $G_2$ is right-differentiable at $\frac{p}{q}$, and the value of the derivative at $\frac{p}{q}$ is $q\re(\tilde{C})$. 

\medskip

As $x\to \frac{p}{q}^-$, we have
\begin{multline}\label{eq:rn-}
\varphi_2(x)=\varphi_2\Big(\frac{p}{q}\Big)-\frac{i \pi^3}{3q^3}(qx-p)\log(|qx-p|)+\Big(\tilde{C}+\frac{\pi^4}{3q^3}\Big)(qx-p)\\ +O((qx-p)^2\log(|qx-p|)). 
\end{multline}
Taking the real part of the both sides of Equation~(\ref{eq:rn-}) shows that $G_2$ is left-differentiable at $\frac{p}{q}$. 
The value of the derivative at $\frac{p}{q}$ is $q\re(\tilde{C})+\frac{\pi^4}{3q^2}$. In particular, $G_2$ is not differentiable at $\frac{p}{q}$. 
At each rational $\frac{p}{q}$, if we denote the left and the right derivative of $G_{2}$ at $\frac{p}{q}$ by $G'_{2}\left({\frac{p}{q}}^{-}\right), G'_{2}\left({\frac{p}{q}}^{+}\right)$ respectively, 
we have $G'_{2}\left({\frac{p}{q}}^{-}\right) - G'_{2}\left({\frac{p}{q}}^{+}\right) = \frac{\pi^4}{3q^2}$.
This completes the proof of the theorem. \hfill \qedsymbol

\section{Proof of Theorem~\ref{2irrIm}}\label{spirr}

\subsection{Properties of continued fractions}

In this section, we will sum up important facts about continued fractions, which we will later use proving Theorem~\ref{2irrIm}. 
For the introduction to continued fractions see a classical textbook by Hardy and Wright \cite{HW}. If not otherwise stated, we will write $a_n=a_n(x)$ and $p_n=p_n(x)$, $q_n=q_n(x)$.

\begin{prop}\label{faqk}
Let $x \in (0,1)\setminus\Q$ and $k \in \N$. We have:
\begin{enumerate}
 \item $\textnormal{Fib}_{k+1}\leq q_k $, where $\textnormal{Fib}_{j}$ is the $j$th Fibonacci number; 
 \item $\sum_{j=0}^{\infty}\frac{1}{q_j} < \infty$; 
 \item $\sum_{j=0}^{k}q_j \leq 3q_k$; 
 \item $\frac{q_k}{2q_{k+1}} \leq T^k(x) \leq  \frac{2q_k}{q_{k+1}}$.
\end{enumerate}
\end{prop}
\begin{proof}
 These properties can be deduced from the definitions, see for example \cite{Kh} and \cite{BM1}.
\end{proof}

Let $\beta_k(x)= \prod_{j=0}^{k}T^{j}(x)$ for $k \geq 0$, and $\beta_{-1}(x)=1$. Let $\gamma_k(x)= \beta_{k-1}(x)\log (\frac{1}{T^k(x)})$, for $k \geq 0$.
Note that for all $k$ and for all $x$,
\begin{equation*}
 0 \leq \beta_k(x) \leq 1 \textnormal{ and } 0 \leq \gamma_k(x).
\end{equation*}

We state the important facts about $\beta_k(x)$ and $\gamma_k(x)$.

\begin{prop}\label{fabag}
 Let $x \in (0,1)\setminus\Q$. We have:
\begin{enumerate}
 \item $\frac{1}{2q_{k+1}} \leq \frac{1}{q_k+q_{k+1}} \leq \beta_k(x) \leq \frac{1}{q_{k+1}}$, for all $k \geq -1$; 
 \item $\frac{\log(q_{k+1})}{q_{k}}-\frac{\log(2q_k)}{q_{k}} \leq \gamma_k(x) \leq \frac{\log(q_{k+1})}{q_{k}}+\frac{\log(2)}{q_{k}}$, for all $k \geq 0$;
 \item $\beta_k(x)=\frac{1}{q_{k+1}+T^{k+1}(x)q_k}$, for all $k \geq -1$.
\end{enumerate}
\end{prop}
\begin{proof}
It follows from the definitions, see \cite[Section 3]{BM1}.
\end{proof}

For the purpose of this paper, we will call the open interval 
defined by the endpoints $[0; b_1, b_2,..., b_k]$ and $[b_1, b_2,..., b_k+1]$ a basic interval on the $k$th level $I(b_1, b_2,..., b_k)$. 
The order depends on the parity of $k$. We will write $I_k(x)$ for the basic interval on the $k$th level that contains $x$.
For all $x \in (0,1)\setminus \Q$ and all $k\in\N$ there exists exactly one basic interval on the $k$th level that contains $x$. 
We will now summarise some observations concerning the basic intervals.
\begin{prop}\label{fabi}
 Let $x\in(0,1)\setminus\Q$. Then we have
\begin{enumerate}
 \item the functions $T^i(x)$, $\beta_i(x)$, $\log(T^i(x))$, $\gamma_i(x)$ are continuous and differentiable on $I_k(x)$ for all $i \leq k$;
 \item for all $x \in I_k(x)$
      \begin{enumerate}
       \item $T^k(x)=\frac{q_kx-p_k}{-q_{k-1}x+p_{k-1}}$;
       \item $\beta_k(x)=(-1)^{k-1}(p_k-q_kx)$;
       \item $\beta_{k-1}(x)=\frac{1}{q_k}(1- q_{k-1}\beta_k(x))$;
       \item $(T^k(x))'=\frac{(-1)^k}{\beta_{k-1}(x)^2}$.
      \end{enumerate}
\end{enumerate}
\end{prop}
\begin{proof}
 This follows from the definitions. See \cite[Section 5.5]{BM1} and \cite[Section 1.1]{R}.
\end{proof}

We can relate $\beta_k(x)$ to $q_k$ using the following claim. 
\begin{claim}\label{caqkitfs}
 We have $(-1)^k\beta_{k}(x)\sum_{j=0}^{k}(-1)^j\frac{T^j(x)}{\beta_{j}(x)^2}=q_k$, for all $x$ and all $k$.
\end{claim}
\begin{proof}
We proceed by induction. If $k=0$, then $$(-1)^0\beta_{0}(x)\sum_{j=0}^{0}(-1)^j\frac{T^j(x)}{\beta_{j}(x)^2} =\beta_{0}(x)\frac{1}{\beta_{0}(x)\beta_{-1}(x)}=1=q_0,$$ by convention. 
Assume $(-1)^{k-1}\beta_{k-1}(x)\sum_{j=0}^{k-1}(-1)^j\frac{T^j(x)}{\beta_{j}(x)^2}=q_{k-1}$. By Proposition~\ref{fabi} 2.(c) we have
\begin{align*}
(-1)^k\beta_{k}(x)\sum_{j=0}^{k}(-1)^j\frac{T^j(x)}{\beta_{j}(x)^2}&
=(-1)^k\frac{1-q_k\beta_{k-1}}{q_{k-1}}\sum_{j=0}^{k-1}(-1)^j\frac{T^j(x)}{\beta_{j}(x)^2}+\frac{(-1)^k\beta_{k}(x)(-1)^k}{\beta_{k}(x)\beta_{k-1}(x)}\\
&=(-1)^k\frac{1-q_k\beta_{k-1}}{q_{k-1}}(-1)^{k-1}\frac{q_{k-1}}{\beta_{k-1}(x)}+\frac{1}{\beta_{k-1}(x)}\\
&=-\frac{1}{\beta_{k-1}(x)}+q_k+\frac{1}{\beta_{k-1}(x)}=q_k.
\end{align*}
This completes the proof of the claim.
\end{proof}

\subsection{Functional equations for $F_2$ and $G_2$} 

We have the following proposition. 
\begin{prop}\label{funeq2}
 Let $x\in(0,1)$. We have
\begin{align}
 F_2(x)&=-x^4 F_2(T(x))-\frac{\pi^3}{3}x\log(x)+P(x)-6\int_{0}^{x}{t^2(x-2t)F_2(T(x))dt},\label{F2feqT}\\
G_2(x)&=x^4 G_2(T(x))-\pi^2x^2\log(x)+Q(x)+6\int_{0}^{x}{t^2(x-2t)G_2(T(x))dt}, \label{G2feqT}
\end{align}
where $P(x), Q(x) \in \R[x]$ are polynomials of degree less than or equal to 3.
\end{prop}
\begin{proof}
 Let $x\in(0,1)$. We apply Proposition~\ref{rnprop} with $\gamma=\bigl(\begin{smallmatrix}
0&-1\\ 1&0
\end{smallmatrix} \bigr)$. Since $x>0$, $\Log(x)$ becomes the natural logarithm $\log(x)$, and we obtain
\begin{multline}\label{step1}
\varphi_2(x)=x^4 \varphi_2\Big(-\frac{1}{x}\Big)-\frac{i \pi^3}{3}x\log(x)+P_{0}(x)\\
-\pi^2x^2\log(x)+6\int_{0}^{x}{t^2(x-2t)\varphi_2\Big(-\frac{1}{x}\Big)dt},
\end{multline}
where $P_{0}(x)=-\frac{i\pi^3}{18}x^3+\left(\frac{f_{\gamma}}{2}+\frac{3\pi^2}{2}\right)x^2+\varphi_2(0)$ with
$f_\gamma(z)=2i\pi^3$
obtained by evaluating (\ref{eq:fgamma}) at $z=i$.
We take imaginary and real parts of Equation~(\ref{step1}) respectively and we get
\begin{align*}
 F_2(x)&=x^4 F_2\Big(-\frac{1}{x}\Big)-\frac{\pi^3}{3}x\log(x)+\im(P_{0})(x)
+6\int_{0}^{x}{t^2(x-2t)F_2\Big(-\frac{1}{x}\Big)dt},\\
G_2(x)&=x^4 G_2\Big(-\frac{1}{x}\Big)+\re(P_{0})(x)
-\pi^2x^2\log(x)+6\int_{0}^{x}{t^2(x-2t)G_2\Big(-\frac{1}{x}\Big)dt}.
\end{align*}
Write $P=\im(P_{0})$, $Q=\re(P_{0})$.
We conclude by observing that since $F_2$ is odd and $G_2$ is even, and they are both $1$-periodic we have that $F_2\big(-\frac{1}{x}\big)=-F_2\big(\frac{1}{x}\big)=-F_2(T(x))$ 
and $G_2\big(-\frac{1}{x}\big)=G_2\big(\frac{1}{x}\big)=G_2(T(x))$. 
\end{proof}

We iterate Equations (\ref{F2feqT}) and (\ref{G2feqT}) to obtain:
\begin{cor}\label{cor1}
 For all $n \in \N^*$ and $x\in(0,1)\setminus \Q$ we have:
\begin{align}
F_2(x)=&(-1)^nF_2(T^{n}(x))\beta_{n-1}(x)^4
+\frac{\pi^3}{3}\sum_{k=0}^{n}(-1)^{k}\beta_{k-1}(x)^2\beta_{k}(x)\gamma_{k}(x)\notag\\
&+\sum_{k=0}^{n}(-1)^{k}P(T^{k}(x))\beta_{k-1}(x)^4\notag\\
&+6\sum_{k=0}^{n}(-1)^{k+1} \beta_{k-1}(x)^4\int_{0}^{T^{k}(x)}{t^2(T^{k}(x)-2t)F_2(T(t))dt} ,\label{fcffen}\\
G_2(x)=&G_2(T^{n}(x))\beta_{n-1}(x)^4+\pi^2\sum_{k=0}^{n}\beta_{k-1}(x)\beta_{k}(x)^2\gamma_{k}(x)\notag\\
&+\sum_{k=0}^{n}Q(T^{k}(x))\beta_{k-1}(x)^4 +6\sum_{k=0}^{n}\beta_{k-1}(x)^4\int_{0}^{T^{k}(x)}{t^2(T^{k}(x)-2t)G_2(T(t))dt} .\label{gcffen}
\end{align}
Letting $n \to \infty$, we get:
\begin{align}
F_2(x)=&\frac{\pi^3}{3}\sum_{k=0}^{\infty}(-1)^{k}\beta_{k-1}(x)^2\beta_{k}(x) \gamma_{k}(x)
+\sum_{k=0}^{\infty}(-1)^{k}P(T^{k}(x))\beta_{k-1}(x)^4 \notag \\
&+6\sum_{k=0}^{\infty}(-1)^{k+1}  \beta_{k-1}(x)^4\int_{0}^{T^{k}(x)}{t^2(T^{k}(x)-2t)F_2(T(t))dt},\label{fcffeninf}\\
G_2(x)=&\pi^2\sum_{k=0}^{\infty}\beta_{k-1}(x)\beta_{k}(x)^2\gamma_{k}(x)+\sum_{k=0}^{\infty}Q(T^{k}(x))\beta_{k-1}(x)^4 \notag\\
&+6\sum_{k=1}^{\infty}\beta_{k-1}(x)^4\int_{0}^{T^{k}(x)}{t^2(T^{k}(x)-2t)G_2(T(t))dt}.\label{gcffeninf}
\end{align}
\end{cor}
\begin{proof}
Equations (\ref{fcffen}) and (\ref{gcffen}) follow from iterating (\ref{F2feqT}) and (\ref{G2feqT}), respectively. Since $|F_2|$ and $|G_2|$ are bounded on $\R$, we have that the terms $|(-1)^nF_2(T^{n}(x))\beta_{n-1}(x)^4|$ 
and $|G_2(T^{n}(x))\beta_{n-1}(x)^4|$ converge to 0 as $n\to \infty$. Thus, 
\begin{align}
F_2(x)=\sum_{k=0}^{\infty}&\Big(\frac{\pi^3}{3}(-1)^{k}\beta_{k-1}(x)^2\beta_{k}(x) \gamma_{k}(x)
+(-1)^{k}P(T^{k}(x))\beta_{k-1}(x)^4 \notag\\
&+(-1)^{k+1} 6 \beta_{k-1}(x)^4\int_{0}^{T^{k}(x)}{t^2(T^{k}(x)-2t)F_2(T(t))dt}\Big),\label{eqfef3}\\
G_2(x)=\sum_{k=0}^{\infty}&\Big(\pi^2\beta_{k-1}(x)\beta_{k}(x)^2\gamma_{k}(x)
+Q(T^{k}(x))\beta_{k-1}(x)^4 \notag\\
&+6\beta_{k-1}(x)^4\int_{0}^{T^{k}(x)}{t^2(T^{k}(x)-2t)G_2(T(t))dt}\Big). \label{eqfeg3}
\end{align}
Finally, we note that $\big|\int_{0}^{T^{k}(x)}{t^2(T^{k}(x)-2t)F_2(T(t))dt}\big|$, $|P(T^{k}(x))|$ are bounded on $[0,1]$, therefore we have
\begin{align*}
\sum_{k=0}^{\infty}\Big|\frac{\pi^3}{3}(-1)^{k}&\beta_{k-1}(x)^2\beta_{k}(x)\gamma_{k}(x)
+(-1)^{k}P(T^{k}(x))\beta_{k-1}(x)^4 \\
&+(-1)^{k+1} 6 \beta_{k-1}(x)^4\int_{0}^{T^{k}(x)}{t^2(T^{k}(x)-2t)F_2(T(t))dt}\Big|\\
\leq&\frac{\pi^3}{3}\sum_{k=0}^{\infty}\beta_{k-1}(x)^2\beta_{k}(x)\gamma_{k}(x)+c_1\sum_{k=0}^{\infty}\beta_{k-1}(x)^4\\
\leq& \frac{2\pi^3}{3}\sum_{k=0}^{\infty}\frac{\log(q_{k+1})}{q_k^3q_{k+1}}+c_1\sum_{k=0}^{\infty}\frac{1}{q_k^4} 
 \quad \quad \quad \quad \quad \quad \quad 
\textnormal{ by Proposition~\ref{fabag} (1) and (2)}\\
 \leq& c_2 \sum_{k=0}^{\infty}\frac{1}{q_k^3}\leq c_2 \sum_{k=1}^{\infty}\frac{1}{\text{Fib}_k^3}  
\quad \quad \quad \quad \quad \quad \quad \quad \quad \quad \quad \quad \quad
\textnormal{ by Proposition~\ref{faqk}~(1)},
\end{align*}
for some constants $c_1$ and $c_2$. This shows that the series (\ref{eqfef3}) converges absolutely and we can change the order of summation obtaining (\ref{fcffeninf}). 
In a similar way, we can show that (\ref{eqfeg3}) converges absolutely and we have (\ref{gcffeninf}). This completes the proof of the corollary.
\end{proof}

\subsection{Proof of Theorem~\ref{2irrIm} (i)}

Let $x\in\R\setminus\Q$. Since $F_2$ is $1$-periodic, we can assume $x\in(0,1)$. For brevity, let
\begin{align}
 u_{1,k}(x)=& (-1)^{k}\beta_{k-1}(x)^2\beta_k(x)\gamma_{k}(x)\notag\\
 u_{2,k}(x)=& (-1)^{k}P(T^{k}(x))\beta_{k-1}(x)^4 \label{u1u2u3}\\
 u_{3,k}(x)=& (-1)^{k+1} \beta_{k-1}(x)^4\int_{0}^{T^{k}(x)}{t^2(T^{k}(x)-2t)F_2(T(t))dt}.\notag
\end{align}
With this notation, we have
\begin{align*}
 F_2(x)&=(-1)^nF_2(T^{n}(x))\beta_{n-1}(x)^4+\frac{\pi^3}{3}\sum_{k=0}^{n}u_{1,k}(x)+\sum_{k=0}^{n}u_{2,k}(x)+6\sum_{k=0}^{n}u_{3,k}(x)\\
&=\frac{\pi^3}{3}\sum_{k=0}^{\infty}u_{1,k}(x)+\sum_{k=0}^{\infty}u_{2,k}(x)+6\sum_{k=0}^{\infty}u_{3,k}(x).
\end{align*}

We are interested in the limit $\frac{F_2(x+h)-F_2(x)}{h}$ as $h\to \infty$. For each $h$, let $K_h \in \N$ such that $x+h\in I_k(x)$ for all $k\leq K_h$ and $x+h \notin I_{K_h+1}(x)$. We make the following observation.
\begin{lem}\label{hitoqkh}
 Let $x \in (0,1)\setminus\Q, |h|>0$ and $K_h$ defined as above, then
 \begin{equation*}
  \frac{1}{2q_{K_h+2}q_{K_h+3}} \leq |h| \leq \frac{2}{q_{K_h}^2}.
 \end{equation*}
If $a_k=1$ only for finitely many indices $k$, then there exists $h_0>0$ such that if $|h|\leq h_0$ we have
 \begin{equation}\label{Khest2}
  \frac{1}{2q_{K_h+1}q_{K_h+2}} \leq |h| \leq \frac{2}{q_{K_h}^2}.
 \end{equation}
\end{lem}
\begin{proof}
Since $x+h \in I_{K_h}(x)$, $|h|$ must be smaller than or equal to the distance from $x$ to one of the endpoints of $I_{K_h}(x)$, which are $\frac{p_{K_h}}{q_{K_h}}$
and $\frac{p_{K_h}+p_{K_h-1}}{q_{K_h}+q_{K_h-1}}$. We then have
\begin{align*}
 |h|\leq&\max\left(\left|x-\frac{p_{K_h}}{q_{K_h}}\right|, \left|x-\frac{p_{K_h}+p_{K_h-1}}{q_{K_h}+q_{K_h-1}}\right| \right)\\
= & \max\left(\frac{\beta_{K_h}(x)}{q_{K_h}},\frac{\beta_{K_h+1}(x)}{q_{K_h+1}}+\frac{a_{K_h+1}-1}{q_{K_h+1}(q_{K_h}+q_{K_h-1})}\right)\\
\leq & \max\left(\frac{1}{q_{K_h}q_{K_h+1}}, \frac{1}{q_{K_h+1}q_{K_h+2}}+\frac{1}{(q_{K_h}+q_{K_h-1})(q_{K_h}+q_{K_h-1})} \right) \textnormal{ by Proposition~\ref{fabag}~(1)}\\
\leq & \frac{2}{q_{K_h}^2}.
\end{align*}

On the other hand, since  $x+h \notin I_{K_h+1}(x)$, $|h|$ must be greater than the distance from $x$ to the boundary of $I_{K_h+1}(x)$. 
By \cite[Proposition 4]{BM1}, $|h| \geq \frac{1}{2q_{K_h+2}q_{K_h+3}}$.
If $a_k=1$ only for finitely many indices $k$, then there exists $h_0>0$ such that for all $|h| \leq h_0$, for all $k\geq K_h$ we have $a_k>1$. 
Then the distance from $x$ the boundary of $I_{K_h+1}(x)$ is greater than or equal to $\frac{1}{2q_{K_h+1}q_{K_h+2}}$, by \cite[Proposition 4]{BM1}.
\end{proof}

\begin{rmk}
 We cannot improve the lower bound on $|h|$ without imposing further conditions on $x$. To illustrate it, we show that we do not even have (\ref{Khest2}) in a general case. 
Let $x$ a square-Brjuno number such that it has infinitely many continued fraction quotients equal to 1 and infinitely 
many different than 1, then there exists a sequence $(h_{K_n})_n$ such that $h_{K_n}\to 0$, as $n\to \infty$, $x+h_{K_n} \in I_{K_n}(x)$, $x+h_{K_n} \notin I_{K_n+1}(x)$
and $|h_{K_n}| \leq \frac{1}{q_{K_n+2}q_{K_n+3}}$.
Indeed, let $K_n$ such that: (1) $K_1$ is the smallest possible and $K_{n+1} > K_n$; (2) $a_{K_n+2}=1$ and $a_{K_n+3}\neq 1$.
Then let $|h_{K_n}|>0$ such that $x+h_{K_n}= \frac{p_{K_n+1}+p_{K_n}}{q_{K_n+1}+q_{K_n}}$.
We have that $|h_{K_n}|\to 0$ as $n\to\infty$; $x+h_{K_n} \in I_{K_n}(x)$, $x+h_{K_n} \notin I_{K_n+1}(x)$; 
and $|h_{K_n}| \leq \frac{1}{q_{K_n+2}q_{K_n+3}}$.
\end{rmk}

\medskip

By Corollary \ref{cor1} we have
\begin{multline}\label{edft}
 \frac{F_2(x+h)-F_2(x)}{h}\\
=\frac{(-1)^{K_h-1}\left(F_2(T^{K_h-1}(x+h))\beta_{K_h-2}(x+h)^4-F_2(T^{K_h-1}(x))\beta_{K_h-2}(x)^4\right)}{h}\\
+\frac{\frac{\pi^3}{3}\sum_{k=0}^{K_h-1}\left(u_{1,k}(x+h)-u_{1,k}(x)\right)}{h}
+\frac{\sum_{k=0}^{K_h-1}\left(u_{2,k}(x+h)-u_{2,k}(x)\right)}{h}\\
+\frac{6\sum_{k=0}^{K_h-1}\left(u_{3,k}(x+h)-u_{3,k}(x)\right)}{h}.
\end{multline}

In the next lemmas, we evaluate the limit of each term as $h\to 0$.

\bigskip

\begin{lem}\label{ftildeh}
Let $x \in (0,1)\setminus\Q$ such that it satisfies (\ref{xtilde}) or (\ref{xtildere}), then 
$$\frac{(-1)^{K_h-1}\left(F_2(T^{K_h-1}(x+h))\beta_{K_h-2}(x+h)^4-F_2(T^{K_h-1}(x))\beta_{K_h-2}(x)^4\right)}{h} \to 0,$$ as $h \to 0$.
\end{lem} 
\begin{proof}
 We have the following
\begin{multline}\label{lem1maineq}
 \frac{(-1)^{K_h-1}\left(F_2(T^{K_h-1}(x+h))\beta_{K_h-2}(x+h)^4-F_2(T^{K_h-1}(x))\beta_{K_h-2}(x)^4\right)}{h}\\
=(-1)^{K_h-1}\beta_{K_h-2}(x)^4\left(\frac{F_2(T^{K_h-1}(x+h))-F_2(T^{K_h-1}(x))}{h}\right)\\
+(-1)^{K_h-1}\left(\frac{\beta_{K_h-2}(x+h)^4-\beta_{K_h-2}(x)^4}{h}F_2(T^{K_h-1}(x+h))\right).
\end{multline}
Firstly, we consider the first summand. We have
\begin{multline*}
 \left|\frac{F_2(T^{K_h-1}(x+h))-F_2(T^{K_h-1}(x))}{h}\right|\\
=\frac{|\sum_{n=1}^\infty \frac{\sigma_1(n)}{n^3} (\sin(2\pi n T^{K_h-1}(x+h))-\sin(2\pi n T^{K_h-1}(x)))|}{|h|}\\
=2\frac{|\sum_{n=1}^\infty \frac{\sigma_1(n)}{n^3} (\sin((T^{K_h-1}(x+h)-T^{K_h-1}(x))\pi n )\cos((T^{K_h-1}(x+h)+T^{K_h-1}(x))\pi n ))|}{|h|}.
\end{multline*}
Let $N= \lceil\frac{1}{h^2}\rceil$, then we have
\begin{align}
\Big|&\frac{F_2(T^{K_h-1}(x+h))-F_2(T^{K_h-1}(x))}{h}\Big|\notag\\
&\leq 2\frac{\sum_{n=1}^N \frac{\sigma_1(n)}{n^3} |\sin((T^{K_h-1}(x+h)-T^{K_h-1}(x))\pi n )|}{|h|}+2\frac{\sum_{n=N+1}^\infty \frac{\sigma_1(n)}{n^3} }{|h|}\notag\\
&\leq 2\pi\frac{\sum_{n=1}^N \frac{\sigma_1(n)}{n^2} |T^{K_h-1}(x+h)-T^{K_h-1}(x)|}{|h|}+2\frac{\sum_{n=N+1}^\infty \frac{\sigma_1(n)}{n^3} }{|h|}\notag\\
&\leq 8\pi q_{K_h-1}^2 \sum_{n=1}^N \frac{\sigma_1(n)}{n^2}+2\frac{\sum_{n=N+1}^\infty \frac{\sigma_1(n)}{n^3} }{|h|}. \label{eqtosubst}
\end{align}
The last line follows from the fact that $T^{K_h-1}$ is continuous and differentiable on $I_{K_h}(x)$, and by the Mean Value Theorem 
$\frac{|T^{K_h-1}(x+h)-T^{K_h-1}(x)|}{|h|}=|(T^{K_h-1}(t))'|$ for some $t$ between $x$ and $x+h$.
By Proposition~\ref{fabi}~(2.d) we have that $(T^{k}(y))'=(-1)^k\beta_{k-1}(y)^{-2}$. By Proposition~\ref{fabag}~(1)
we conclude that $|(T^{K_h-1}(t))'|\leq 4q_{K_h-1}^2$. 

\medskip 

Consider $\sum_{n=1}^N \frac{\sigma_1(n)}{n^2}$, by Abel's summation formula
\begin{align*}
 \sum_{n=1}^N \frac{\sigma_1(n)}{n^2}=& \sum_{n=1}^{N-1}\left(\frac{1}{n^2}-\frac{1}{(n+1)^2} \right)\sum_{k=1}^{n}\sigma_1(k)+\frac{1}{N^2}\sum_{n=1}^{N}\sigma_1(n)\\
\leq & 3\sum_{n=1}^{N-1}\frac{1}{n^3} \sum_{k=1}^{n}\sigma_1(k)+\frac{1}{N^2}\sum_{n=1}^{N}\sigma_1(n).
\end{align*}
By Theorem 3 in \cite[p.40]{T}, there exists $c_1 > 0$ such that $\sum_{j=1}^{k}\sigma_1(j) \leq \frac{\pi^2}{12}k^2+c_1 k\log k$ for all $k\in\N$, and we have
\begin{align}
 \sum_{n=1}^N \frac{\sigma_1(n)}{n^2}
\leq & \frac{\pi^2}{4}\sum_{n=1}^{N-1}\frac{1}{n}+ 3c_1\sum_{n=1}^{N-1}\frac{\log n}{n^2} +\frac{\pi^2}{12}+c_1\frac{\log N}{N}\notag\\
\leq & \left(\frac{\pi^2}{4}+3c_1\right)\sum_{n=1}^{N-1}\frac{1}{n}+\frac{\pi^2}{12}+c_1\frac{\log N}{N}\leq c_2 \log N, \label{sumsigma2}
\end{align}
for some constant $c_2>0$, as $\sum_{n=1}^{k}\frac{1}{k} \leq \log(k)+2$ for all $k\in\N$.

\medskip

Consider $\sum_{n=1}^N \frac{\sigma_1(n)}{n^3}$. By \cite[p.88]{T}, we have $\sigma_1(k) \leq c_3k\log(\log(k))$ for some constant $c_3$ for all $k\in\N$. we have
\begin{align}
\sum_{n=N+1}^\infty \frac{\sigma_1(n)}{n^3} \leq & c_3\sum_{n=N+1}^\infty \frac{\log(\log(n))}{n^2}\leq c_3\sum_{n=N+1}^\infty \frac{(\log(n))^{1/2}}{n^2} c_3\sum_{n=N+1}^\infty \frac{1}{n^{7/4}}\notag\\
\leq &\frac{c_3}{(N+1)^{7/4}}+c_3\int_{N+1}^\infty\frac{1}{x^{7/4}}dx=\frac{c_3}{(N+1)^{7/4}}+\frac{4c_3}{3(N+1)^{3/4}}\leq \frac{7c_3}{3N^{3/4}}. \label{sumsigma3}
\end{align}

Assume $|h|<1$. Substituting (\ref{sumsigma2}) and (\ref{sumsigma3}) into (\ref{eqtosubst}), we get
\begin{align*}
\Big|\frac{F_2(T^{K_h-1}(x+h))-F_2(T^{K_h-1}(x))}{h}\Big|
&\leq 8\pi c_2 q_{K_h-1}^2 \log N+\frac{14c_3}{3N^{3/4}|h|}\\
&\leq 8\pi c_2 q_{K_h-1}^2 \log\Big(\frac{2}{h}\Big)+\frac{14}{3}|h|^{1/2},
\end{align*}
by the choice of $N$.

\medskip

By Lemma~\ref{hitoqkh} and Proposition~\ref{fabag}~(1), we have
\begin{align*}
 \left|\frac{F_2(T^{K_h-1}(x+h))-F_2(T^{K_h-1}(x))}{h}\right| \beta_{K_h-2}(x)^4 \leq & \frac{8\pi c_2}{q_{K_h-1}^2} \log (4q_{K_h+2}q_{K_h+3})+\frac{14}{3{q_{K_h-1}^4}}|h|^{1/2}\\
\leq & c_3 \frac{\log (q_{K_h+3})}{q_{K_h-1}^2}+\frac{14}{3{q_{K_h-1}^4}}|h|^{1/2},
\end{align*}
for some constant $c_3>0$. If $x$ satisfies (\ref{xtilde}), it converges to $0$ as $h\to 0$. If $x$ satisfies (\ref{xtildere}), then  Lemma~\ref{hitoqkh} and Proposition~\ref{fabag}~(1) imply
\begin{equation*}
 \left|\frac{F_2(T^{K_h-1}(x+h))-F_2(T^{K_h-1}(x))}{h}\right| \beta_{K_h-2}(x)^4 \leq c_4 \frac{\log (q_{K_h+2})}{q_{K_h-1}^2}+\frac{14}{3{q_{K_h-1}^4}}|h|^{1/2},
\end{equation*}
for some constant $c_4>0$, and it also converges to $0$.

\medskip

Finally, we consider the second summand of (\ref{lem1maineq}). Since the function 
$\beta_{K_h-2}(y)^4$ is continuous and differentiable on 
$I_{K_h}(x)$, the Mean Value Theorem implies that for some $t$ between $x$ and $x+h$ we have
\begin{align*}
 \frac{|\beta_{K_h-2}(x+h)^4-\beta_{K_h-2}(x)^4|}{|h|}&=|(\beta_{K_h-2}(t)^4)'|\\
&=4\beta_{K_h-2}(t)^3 (-1)^{K_h-2}q_{K_h-2} &\textnormal{ by Proposition~\ref{fabi}~(2.b)}\\
& \leq \frac{4}{q_{K_h-1}^2}&\textnormal{ by Proposition~\ref{fabag}~(1).}
\end{align*}
 Observing that $|F_2|$ is bounded and $\|F_{2}\|_\infty=\sup_{y\in[0,1)}|F_{2,3}(y)|$
 we obtain
\begin{equation*}
 \left|\frac{\beta_{K_h-1}(x+h)^4-\beta_{K_h-1}(x)^4}{h}F_2(T^{K_h}(x+h))\right| \leq \frac{4\|F_{2}\|_\infty}{q_{K_h-1}^2},
\end{equation*}
which converges to 0 as $h \to 0$ for all $x\in (0,1)\setminus\Q$. This completes the proof of Lemma~\ref{ftildeh}.
\end{proof}

\bigskip

\begin{lem}\label{u2lim}
Let $x \in (0,1)\setminus\Q$ be a square-Brjuno number, then 
\begin{multline*}
\frac{\sum_{k=0}^{K_h-1}\left(u_{1,k}(x+h)-u_{1,k}(x)\right)}{h}\\
 \to \sum_{k=0}^{\infty}\beta_{k-1}(x)\gamma_{k}(x)+4\sum_{k=0}^{\infty}\Big((-1)^{k}\beta_{k-1}(x)^2\beta_{k}(x)\gamma_k(x)\sum_{j=0}^{k-1}(-1)^j\frac{T^j(x)}{\beta_{j}(x)^2}\Big)-\sum_{k=0}^{\infty}\beta_{k-1}(x)^2
\end{multline*}
 as $h \to 0$.
\end{lem}
First we will establish the following two lemmas, which we will use in the proof of Lemma~\ref{u2lim}.

\begin{lem}\label{gambeta}
 Let $x \in (0,1)\setminus\Q$. The series $\sum_{k=0}^{\infty}\beta_{k-1}(x)\gamma_k(x)$ converges if and only if $\sum_{k=0}^{\infty}\frac{\log(q_{k+1})}{q_k^2}$ converges.
\end{lem}
\begin{proof}
 Since $\sum_{k=0}^{\infty}\beta_{k-1}(x)\gamma_k(x)$ is positive, we have:
\begin{equation*}
\sum_{k=0}^{\infty}\beta_{k-1}(x)\gamma_k(x) \leq \sum_{k=0}^{\infty}\frac{\log(2q_{k+1})}{q_k^2}\leq \sum_{k=0}^{\infty}\frac{\log(2)}{q_k^2}+\sum_{k=0}^{\infty}\frac{\log(q_{k+1})}{q_k^2},\\
\end{equation*}
where the first inequality follows from Proposition~\ref{fabag}. Since $\sum_{k=0}^{\infty}\frac{\log(2)}{q_k^2}$ converges for all $x \in (0,1)\setminus\Q$, if $\sum_{k=0}^{\infty}\frac{\log(q_{k+1})}{q_k^2}$ converges, 
then $\sum_{k=0}^{\infty}\beta_{k-1}(x)\gamma_k(x)$ converges as well. 
For the inverse note that:
\begin{align*}
\sum_{k=0}^{\infty}&\frac{\log(q_{k+1})}{q_k^2} \leq  \sum_{k=0}^{\infty}\frac{\gamma_k(x)}{q_k} +\sum_{k=0}^{\infty}\frac{\log(2q_k)}{q_k^2}
 \quad \quad \quad \quad \quad \quad \quad \quad \quad \quad \quad \textnormal{by Proposition~\ref{fabag}~(2)}\\
= & \sum_{k=0}^{\infty}\beta_{k-1}(x)\gamma_k(x) +\sum_{k=0}^{\infty}\beta_{k-1}(x)\gamma_k(x) T^{k}(x)\frac{q_{k-1}}{q_k}+\sum_{k=0}^{\infty}\frac{\log(2q_k)}{q_k^2} \textnormal{ by Proposition~\ref{fabag}~(3)}\\
\leq & 2\sum_{k=0}^{\infty}\beta_{k-1}(x)\gamma_k(x)+\sum_{k=0}^{\infty}\frac{\log(2q_k)}{q_k^2},
\end{align*}
as $T^{k}(x)\frac{q_{k-1}}{q_k} \leq 1$.
The sum $\sum_{k=0}^{\infty}\frac{\log(2q_k)}{q_k^2}$ converges for all $x$, which completes the proof of the lemma.
\end{proof}

\begin{lem}\label{partisumlem}
 The series 
$$4\sum_{k=0}^{\infty}\Big((-1)^{k}\beta_{k-1}(x)^2\beta_{k}(x)\gamma_k(x)\sum_{j=0}^{k-1}(-1)^j\frac{T^j(x)}{\beta_{j}(x)^2}-\beta_{k-1}(x)^2\Big)$$ converges for all $x \in (0,1)\setminus\Q$.
\end{lem}
\begin{proof}
 By Claim~\ref{caqkitfs}, we have 
\begin{align*}
\Bigg|4\sum_{k=0}^{\infty}\Big((-1)^{k}\beta_{k-1}(x)^2&\beta_{k}(x)\gamma_k(x)\sum_{j=0}^{k-1}(-1)^j\frac{T^j(x)}{\beta_{j}(x)^2}-\beta_{k-1}(x)^2\Big)\Bigg| \\
&\leq 4\sum_{k=0}^{\infty}|\beta_{k-1}(x)\beta_{k}(x)\gamma_k(x)q_{k-1}|+\sum_{k=0}^{\infty}\beta_{k-1}(x)^2\\
&\leq 4\sum_{k=0}^{\infty}\frac{q_{k-1}}{q_{k}q_{k+1}}\frac{\log(2q_{k+1})}{q_k} +\sum_{k=0}^{\infty}\frac{1}{q_k^2}\textnormal{ by Proposition~\ref{fabag} (1) and (2)}\\
&\leq 9\sum_{k=0}^{\infty}\frac{1}{q_{k}},
\end{align*}
which converges by Proposition~\ref{faqk}~(2).
\end{proof}

\medskip

\noindent\textit{Proof of Lemma~\ref{u2lim}.}
Let $x \in (0,1)\setminus\Q$ be square-Brjuno. By Proposition~\ref{fabi}~(2), we have
\begin{multline}\label{eq:bigeq}
\Bigg| \frac{\sum_{k=0}^{K_h-1}\left(u_{1,k}(x+h)-u_{1,k}(x)\right)}{h}
 -\sum_{k=0}^{\infty}\beta_{k-1}(x)\gamma_{k}(x)\\
-4\sum_{k=0}^{\infty}\Big((-1)^{k}\beta_{k-1}(x)^2\beta_{k}(x)\gamma_k(x)\sum_{j=0}^{k-1}(-1)^j\frac{T^j(x)}{\beta_{j}(x)^2}\Big) +\sum_{k=0}^{\infty}\beta_{k-1}(x)^2\Bigg|\\
\leq  \Big|-\sum_{k=0}^{K_h-1}\beta_{k-1}(x)^2\log(T^{k}(x+h))-\sum_{k=0}^{K_h-1}\beta_{k-1}(x)\gamma_{k}(x)\Big|\\
+4\Big|\sum_{k=0}^{K_h-1}(-1)^{k+1}\Big(\beta_{k-1}(x)^3\beta_{k}(x)\log(T^{k}(x+h))\sum_{j=0}^{k-1}(-1)^j\frac{T^j(x)}{\beta_{j}(x)^2}\Big)\\ 
-\sum_{k=0}^{K_h-1}\Big((-1)^{k}\beta_{k-1}(x)^2\beta_{k}(x)\gamma_k(x)\sum_{j=0}^{k-1}(-1)^j\frac{T^j(x)}{\beta_{j}(x)^2}\Big)\Big|\\
+\Big|\sum_{k=0}^{K_h-1}(-1)^{k+1}\frac{\beta_{k-1}(x)^3\beta_k(x)(\log(T^{k}(x+h))-\log(T^{k}(x)))}{h}+\sum_{k=0}^{K_h-1}\beta_{k-1}(x)^2\Big|\\
+\Big|\sum_{k=0}^{K_h-1}(-1)^{k+1}A_k h\log(T^{k}(x+h))\Big|
+\Big|\sum_{k=0}^{K_h-1}(-1)^{k+1}B_k h^2\log(T^{k}(x+h))\Big|\\
+\Big|\sum_{k=0}^{K_h-1}(-1)^{k+1}C_k h^3\log(T^{k}(x+h))\Big|
+ \Big| -\sum_{k=K_h}^{\infty}\beta_{k-1}(x)\gamma_{k}(x)\\
-4\sum_{k=K_h}^{\infty}\Big((-1)^{k}\beta_{k-1}(x)^2\beta_{k}(x)\gamma_k(x)\sum_{j=0}^{k-1}(-1)^j\frac{T^j(x)}{\beta_{j}(x)^2}\Big) +\sum_{k=K_h}^{\infty}\beta_{k-1}(x)^2\Big|,
\end{multline}
with
\begin{align}
 A_k&=-3\beta_{k-1}(x)^2q_{k-1}q_k+6\beta_{k-1}(x)\beta_k(x)q_{k-1}+3(-1)^k \beta_{k-1}(x)^2 \beta_k(x)q_{k-1}^2\notag\\
 B_k&=(-1)^k( 3 \beta_{k-1}(x)q_{k-1}^2q_k-3 \beta_k(x)q_{k-1}^2- \beta_k(x)q_{k-1}^3)-3\beta_{k-1}(x)\beta_k(x)q_{k-1}^3\label{defabc}\\
 C_k&=-q_{k-1}^3q_k.\notag
\end{align}
By Lemmas \ref{gambeta} and \ref{partisumlem}, the last term converges to 0 as $h\to 0$. We will now show that all the other terms also converge to 0.

\medskip

We observe that by Proposition~\ref{fabi}, for all $k \leq K_h$ the function $T^k$ is non-zero, continuous, and differentiable on $I_k(x)$, 
 hence $\log(T^k)$ is continuous, and differentiable on $I_k(x)$.
Then for all $k\leq K_h-1$ and $y \in I_k(x)$ we have $\log(T^k(y))'=\frac{(-1)^k}{T^k(y)\beta_{k-1}(y)^2}$.
By the Mean Value Theorem
$|\log(T^{k}(x+h))-\log(T^{k}(x))|=|h| \frac{1}{T^k(t_k)}\frac{1}{\beta_{k-1}(t_k)^2},$
for some $t_k$ between $x$ and $x+h$. Since $t_k \in I_k(x)$, by Proposition~\ref{faqk}~(4) and \ref{fabag}~(1), we have 
$\frac{1}{T^k(t_k)}\frac{1}{\beta_{k-1}(t_k)^2} \leq \frac{2q_{k+1}}{q_k} 4q_k^2= 8q_{k}q_{k+1}$ and
$|\log(T^{k}(x+h))-\log(T^{k}(x))|\leq 8q_{k}q_{k+1}|h|$. Thus,
\begin{align*}
 \Big|-\sum_{k=0}^{K_h-1}\beta_{k-1}(x)^2&\log(T^{k}(x+h))-\sum_{k=0}^{K_h-1}\beta_{k-1}(x)\gamma_{k}(x)\Big|\\
&=  \Big|\sum_{k=0}^{K_h-1}\beta_{k-1}(x)^2(\log(T^{k}(x+h))-\log(T^{k}(x)))\Big|\\
&\leq 8|h|\sum_{k=0}^{K_h-1}\beta_{k-1}(x)^2q_{k}q_{k+1}\leq  8|h|\sum_{k=0}^{K_h-1}\frac{q_{k+1}}{q_{k}}\leq \frac{16}{q_{K_h}}\sum_{k=0}^{K_h-1}\frac{1}{q_{k}},
\end{align*}
 by Lemma~\ref{hitoqkh}, which converges to 0 as $h\to0$.

\medskip

Using the same arguments and applying Claim~\ref{caqkitfs}, we obtain
\begin{align*}
 4\Big|\sum_{k=0}^{K_h-1}(-1)^{k+1}&\Big(\beta_{k-1}(x)^3\beta_{k}(x)\log(T^{k}(x))\sum_{j=0}^{k-1}(-1)^j\frac{T^j(x)}{\beta_{j}(x)^2}\Big)\\ 
&-\sum_{k=0}^{K_h-1}\Big((-1)^{k}\beta_{k-1}(x)^2\beta_{k}(x)\gamma_k(x)\sum_{j=0}^{k-1}(-1)^j\frac{T^j(x)}{\beta_{j}(x)^2}\Big)\Big|\\
\leq & 4\sum_{k=0}^{K_h-1}\beta_{k-1}(x)^2\beta_{k}(x)q_{k-1}|\log(T^{k}(x+h))q_{k-1}-\log(T^k(x))|\\
\leq & 32\sum_{k=0}^{K_h-1}\frac{q_{k-1}}{q_k}|h|\leq \frac{64}{q_{K_h}}\sum_{k=0}^{K_h-1}\frac{1}{q_k},
\end{align*}
which converges to 0 as $h\to0$.

\medskip

By the Mean Value Theorem, we have
\begin{align*}
 \Big|\sum_{k=0}^{K_h-1}(-1)^{k+1}&\frac{\beta_{k-1}(x)^3\beta_k(x)(\log(T^{k}(x+h))-\log(T^{k}(x)))}{h}+\sum_{k=0}^{K_h-1}\beta_{k-1}(x)^2\Big|\\
=& \Big|-\sum_{k=0}^{K_h-1}\frac{\beta_{k-1}(x)^3\beta_k(x)}{\beta_{k-1}(t_k)\beta_{k}(t_k)}+\sum_{k=0}^{K_h-1}\beta_{k-1}(x)^2\Big|,
\end{align*}
for some $t_k$ between $x$ and $x+h$. 
Also, $\beta_{k-1}(y)\beta_{k}(y)$ is continuous and differentiable on $I_k(x)$ with the derivative $(\beta_{k-1}(y)\beta_{k}(y))'=(-1)^k\beta_{k}(y)q_{k-1} +(-1)^{k-1}\beta_{k-1}(y)q_k$.
By Proposition~\ref{fabag}~(1) for all $y \in I_k(x)$ we have $|(\beta_{k-1}(y)\beta_{k}(y))'|\leq 2$. Therefore, we have 
\begin{align*}
 \Big|\sum_{k=0}^{K_h-1}(-1)^{k+1}&\frac{\beta_{k-1}(x)^3\beta_k(x)(\log(T^{k}(x+h))-\log(T^{k}(x)))}{h}+\sum_{k=0}^{K_h-1}\beta_{k-1}(x)^2\Big|\\
=& \Big|\sum_{k=0}^{K_h-1}\beta_{k-1}(x)^2\frac{-\beta_{k-1}(x)\beta_k(x)+\beta_{k-1}(t_k)\beta_{k}(t_k)}{\beta_{k-1}(t_k)\beta_{k}(t_k)}\Big|\\
\leq& 2\sum_{k=0}^{K_h-1}\beta_{k-1}(x)^2\frac{|x-t_k|}{\beta_{k-1}(t_k)\beta_{k}(t_k)}
\leq 8\sum_{k=0}^{K_h-1}\frac{q_{k+1}|h|}{q_k} \textnormal{ by Proposition~\ref{fabag}~(1)}\\
\leq& \frac{16}{q_{K_h}}\sum_{k=0}^{K_h-1}\frac{1}{q_k},
\end{align*}
by Lemma~\ref{hitoqkh}, which converges to 0 as $h\to0$ by Proposition~\ref{faqk}~(2).

\medskip

By Proposition~\ref{fabag}~(1), we have 
$|A_k|\leq 3\frac{q_{k-1}}{q_k(x)}+6\frac{q_{k-1}}{q_{k}q_{k+1}}+3\frac{q_{k-1}^2}{q_{k}^2q_{k+1}} \leq 12.$
Also by Proposition~\ref{faqk}~(4), $|\log(T^{k}(x+h))| \leq \frac{2q_{k+1}}{q_k}$. Then by Lemma~\ref{hitoqkh}, we have
\begin{equation*}
 \left|\sum_{k=0}^{K_h-1}(-1)^{k+1}A_k h\log(T^{k}(x+h))\right| \leq 24\sum_{k=0}^{K_h-1} \frac{q_{k+1}}{q_k} |h| \leq \frac{48}{q_{K_h}}\sum_{k=0}^{\infty} \frac{1}{q_k},
\end{equation*}
which converges to 0 as $h\to0$ by Proposition~\ref{faqk}~(2).
Similarly, $ |B_k|\leq 3q_{k-1}^2+3\frac{q_{k-1}^2}{q_{k+1}}+3\frac{q_{k-1}^3}{q_{k}q_{k+1}}+\frac{q_{k-1}^3}{q_{k+1}} \leq 10q_{k-1}^2.$ We then have
\begin{equation*}
 \left|\sum_{k=0}^{K_h-1}(-1)^{k+1}B_k h^2\log(T^{k}(x+h))\right| \leq 20\sum_{k=0}^{K_h-1} \frac{q_{k-1}^2 q_{k+1}}{q_k} h^2
\leq \frac{80}{q_{K_h}}\sum_{k=0}^{\infty} \frac{1}{q_k}.
\end{equation*}
which converges to 0 as $h\to0$ by Proposition~\ref{faqk}~(2).
Finally,
\begin{equation*}
 \left|\sum_{k=0}^{K_h-1}(-1)^{k+1}C_k h^3\log(T^{k}(x+h))\right|\leq \sum_{k=0}^{K_h-1} \frac{q_{k-1}^3q_k q_{k+1}}{q_k} |h|^3\leq  \frac{8}{q_{K_h}}\sum_{k=0}^{\infty} \frac{1}{q_k},
\end{equation*}
which converges to 0 as $h\to0$ by Proposition~\ref{faqk}~(2).

This shows that (\ref{eq:bigeq}) converges to 0 as $h\to0$ completing the proof of the lemma.
\hfill\qedsymbol

\bigskip

\begin{lem}\label{u3diff}
Let $x \in (0,1)\setminus\Q$, then 
\begin{multline*}
\frac{\sum_{k=0}^{K_h-1}\left(u_{2,k}(x+h)-u_{2,k}(x)\right)}{h} \\
\to \sum_{k=0}^{\infty}(P(T^k(x)))'\beta_{k-1}(x)^2+4\sum_{k=0}^{\infty}(-1)^{k}P(T^k(x))\beta_{k-1}(x)^4\sum_{j=0}^{k-1}(-1)^j\frac{T^j(x)}{\beta_{j}(x)^2},
\end{multline*}
as $h\to 0$, where $(P(T^k(x)))'$ is the derivative of the polynomial $P$ evaluated at $T^k(x)$.
\end{lem}

Before we start proving Lemma~\ref{u3diff}, we will prove the following lemma, which we will use in proving Lemma~\ref{u3diff}.
\begin{lem}\label{s3sum}
The series $$\sum_{k=0}^{\infty}\Big((P(T^k(x)))'\beta_{k-1}(x)^2+(-1)^{k}4P(T^k(x))\beta_{k-1}(x)^4\sum_{j=0}^{k-1}(-1)^j\frac{T^j(x)}{\beta_{j}(x)^2}\Big)$$ converges for all $x \in (0,1)\setminus\Q$.
\end{lem}
\begin{proof}
Firstly, by Claim~\ref{caqkitfs} we have $q_{k-1}=(-1)^{k-1}\beta_{k-1}(x)\sum_{j=0}^{k-1}(-1)^j\frac{T^j(x)}{\beta_{j}(x)^2}$.
Write $\|P\|_\infty=\sup_{y\in(0,1)}|P(y)|$ and $||P'||_\infty=\sup_{y\in(0,1)}|P(y)'|$.
Since $P$ and $P'$ are polynomials, we have $\|P\|_\infty$ and $\|P'\|_\infty$ are finite.
We then have:
\begin{align*}
\Big|\sum_{k=0}^{\infty}\Big((P(T^k(x)))'\beta_{k-1}(x)^2&+(-1)^{k}4P(T^k(x))\beta_{k-1}(x)^4\sum_{j=0}^{k-1}(-1)^j\frac{T^j(x)}{\beta_{j}(x)^2}\Big)\Big|\\
&\leq \|P'\|_\infty\sum_{k=0}^{\infty}\beta_{k-1}(x)^2+4\|P\|_\infty\sum_{k=0}^{\infty}\beta_{k-1}(x)^3q_{k-1}\\
&\leq \|P'\|_\infty\sum_{k=0}^{\infty}\frac{1}{q_{k-1}^2}+4\|P\|_\infty\sum_{k=0}^{\infty}\frac{q_{k-1}}{q_k^3}\textnormal{ by  Proposition~\ref{fabag}~(1)}\\
& \leq (\|P'\|_\infty+4\|P\|_\infty)\sum_{k=0}^{\infty}\frac{1}{q_{k-1}^2},
\end{align*}
which converges for all $x\in(0,1)\setminus\Q$ by Proposition~\ref{faqk}~(2).
\end{proof}

\medskip

\noindent\textit{Proof of Lemma~\ref{u3diff}.} Let $x\in(0,1)\setminus\Q$. 
 We have $u_{2,k}(x)=(-1)^{k}P(T^{k}(x))\beta_{k-1}(x)^4$
where $P(y)=\hat{A}y^3+\hat{B}y^2+\hat{C}y+\hat{D}$, for some constants $\hat{A}, \hat{B}, \hat{C}, \hat{D} \in \R$. 
An easy but long calculation shows that applying Claim~\ref{caqkitfs}, we have
\begin{multline*}
 \frac{\sum_{k=0}^{K_h-1}\left(u_{2,k}(x+h)-u_{2,k}(x)\right)}{h}\\
=\frac{\sum_{k=0}^{K_h-1}\left((-1)^{k}P(T^{k}(x+h))\beta_{k-1}(x+h)^4-(-1)^{k}P(T^{k}(x))\beta_{k-1}(x)^4\right)}{h}\\
=\sum_{k=0}^{K_h-1}(P(T^k(x)))'\beta_{k-1}(x)^2-4\sum_{k=0}^{K_h-1}P(T^k(x))\beta_{k-1}(x)^3q_{k-1}\\
+\hat{A}S_1(h)+\hat{B}S_2(h)+\hat{C}S_3(h)+\hat{D}S_4(h),
\end{multline*}
where 
\begin{align*}
 S_1(h)=&3h\sum_{k=0}^{K_h-1}(-1)^{k}\beta_{k-1}(x)\beta_k(x)q_k^2+h^2\sum_{k=0}^{K_h-1}\beta_{k-1}(x)q_k^3+3h\sum_{k=0}^{K_h-1}(-1)^{k-1}\beta_k(x)^2q_{k-1}q_k\\
&-3h^2\sum_{k=0}^{K_h-1}\beta_k(x)q_{k-1}q_k^2+h^3\sum_{k=0}^{K_h-1}(-1)^{k-1}q_{k-1}q_k^3\\
 S_2(h)=&h\sum_{k=0}^{K_h-1}(-1)^{k}\beta_{k-1}(x)^2q_k^2-4h\sum_{k=0}^{K_h-1}(-1)^{k}\beta_{k-1}(x)\beta_k(x)q_{k-1}q_k -2h^2\sum_{k=0}^{K_h-1}\beta_{k-1}(x)q_{k-1}q_k^2\\
&+h\sum_{k=0}^{K_h-1}(-1)^{k}\beta_k(x)^2q_{k-1}^2+2h^2\sum_{k=0}^{K_h-1}\beta_k(x)q_{k-1}^2q_k+h^3\sum_{k=0}^{K_h-1}(-1)^{k}q_{k-1}^2q_k^2.\\
 S_3(h)=&3h\sum_{k=0}^{K_h-1}(-1)^{k}\beta_{k-1}(x)\beta_k(x)q_{k-1}^2-h^2\sum_{k=0}^{K_h-1}\beta_k(x)q_{k-1}^3+3h\sum_{k=0}^{K_h-1}(-1)^{k-1}\beta_{k-1}(x)^2q_{k-1}q_k\\
&+3h^2\sum_{k=0}^{K_h-1}\beta_{k-1}(x)q_{k-1}^2q_k+h^3\sum_{k=0}^{K_h-1}(-1)^{k-1}q_{k-1}^3q_k\\
 S_4(h)=&6h\sum_{k=0}^{K_h-1}(-1)^{k}\beta_{k-1}(x)^2q_{k-1}^2-4h^2\sum_{k=0}^{K_h-1}\beta_{k-1}(x)q_{k-1}^3+h^3\sum_{k=0}^{K_h-1}(-1)^{k}q_{k-1}^4,
\end{align*}
and $(P(T^k(x)))'$ is the derivative of the polynomial $P$ evaluated at $T^k(x)$, that is $$(P(T^k(x)))'=3\hat{A}(T^k(x))^2+2\hat{B}T^k(x)+\hat{C}.$$
We then have
\begin{multline}\label{eq:s3sum}
 \Big|\frac{\sum_{k=0}^{K_h-1}\left(u_{2,k}(x+h)-u_{2,k}(x)\right)}{h}-\sum_{k=0}^{\infty}(P(T^k(x)))'\beta_{k-1}(x)^2
+4\sum_{k=0}^{\infty}P(T^k(x))\beta_{k-1}(x)^3q_{k-1}\Big|\\
\leq \Big|\sum_{k=K_h}^{\infty}(P(T^k(x)))'\beta_{k-1}(x)^2-4\sum_{k=K_h}^{\infty}P(T^k(x))\beta_{k-1}(x)^3q_{k-1}\Big|\\
+|\hat{A}S_1(h)|+|\hat{B}S_2(h)|+|\hat{C}S_3(h)|+|\hat{D}S_4(h)|.
\end{multline}
The first term converges to 0 as $h\to 0$ by Lemma~\ref{s3sum}. Then applying Proposition~\ref{fabag}~(1) and Lemma~\ref{hitoqkh}, we obtain
\begin{align*}
 |S_1(h)|+
 |S_2(h)|&+
 |S_3(h)|+
 |S_4(h)|\\
\leq& \frac{22}{q_{K_h}}\sum_{k=0}^{K_h-1}\frac{1}{q_k}+\frac{22}{q_{K_h}}\sum_{k=0}^{K_h-1}\frac{1}{q_k}+\frac{22}{q_{K_h}}\sum_{k=0}^{K_h-1}\frac{1}{q_k}+\frac{22}{q_{K_h}}\sum_{k=0}^{K_h-1}\frac{1}{q_k}\xrightarrow[h\to 0]{} 0.
\end{align*}
It shows that the expression in (\ref{eq:s3sum}) converges to 0 as $h\to0$ which completes the proof of Lemma~\ref{u3diff}.
\hfill\qedsymbol

\bigskip

\begin{lem}\label{u1diff}
 Let $x \in (0,1)\setminus\Q$, then 
\begin{multline*}
 \frac{\sum_{k=0}^{K_h-1}\left(u_{3,k}(x+h)-u_{3,k}(x)\right)}{h}\\
 \to \sum_{k=0}^{\infty} \Big(\beta_{k}(x)^2T^{k}(x)F_2(T^{k+1}(x))
+\beta_{k-1}(x)^2 \int_{0}^{p(k)}{t^2F_2(T(t))dt}\\
+4(-1)^{k+1}\int_{0}^{T^{k}(x)}{t^2(T^{k}(x)-2t)F_2(T(t))dt}\cdot\beta_{k-1}(x)^4\sum_{j=0}^{k-1}(-1)^j\frac{T^j(x)}{\beta_{j}(x)^2}\Big),
\end{multline*}
 as $h \to 0$ where $p(k)$ is the smaller endpoint of the interval $I_k(x)$, that is $p(k)=\frac{p_k}{q_k}$ if $k$ is even, and  $p(k)=\frac{p_k+p_{k-1}}{q_k+q_{k-1}}$ if $k$ is odd.
\end{lem}
Before proving Lemma~\ref{u1diff} we will prove some claims and lemmas, which then we will use in the proof of Lemma~\ref{u1diff}. 
First note that for all $k \leq K_h$ the function $u_{3,k}$ is continuous and differentiable on $I_{K_h}(x)$. 
For brevity, write $\mathcal{I}_k(x)=\int_{0}^{T^{k}(x)}{t^2(T^{k}(x)-2t)F_2(T(t))dt}$. 
We will now calculate the derivative of $u_{1,k}$. 
We begin by calculating the derivative of $\mathcal{I}_k(x)$.
\begin{claim}\label{claim:intikx}
 Let $x \in (0,1)\setminus\Q$. For all $k\in\N$ we have $$\mathcal{I}_k'(x)=\frac{(-1)^{k+1}}{\beta_{k-1}(x)^2}\int_{0}^{p(k)}{t^2F_2(T(t))dt}+\frac{(-1)^{k+1}}{\beta_{k-1}(x)^2}T^{k}(x)^3F_2(T^{k+1}(x)),$$
where $p(k) \in \Q$ is the smaller endpoint of the interval $I_k(x)$, that is $p(k)=\frac{p_k}{q_k}$ if $k$ is even, and  $p(k)=\frac{p_k+p_{k-1}}{q_k+q_{k-1}}$ if $k$ is odd. 
\end{claim}
\begin{proof}
 We use the substitution $y=T^k(x)$, hence $\frac{dy}{dx}=\frac{(-1)^k}{\beta_{k-1}(x)^2}$, and we have
\begin{align*}
 \mathcal{I}_k'(x)=&\frac{d}{dx}\int_{0}^{T^{k}(x)}{t^2(T^{k}(x)-2t)F_2(T(t))dt}=\frac{(-1)^k}{\beta_{k-1}(x)^2}\frac{d}{dy}\int_{0}^{y}{t^2(y-2t)F_2(T(t))dt}\\
=&\frac{(-1)^k}{\beta_{k-1}(x)^2}\frac{d}{dy}\int_{0}^{p(k)}{t^2(y-2t)F_2(T(t))dt}+\frac{(-1)^k}{\beta_{k-1}(x)^2}\frac{d}{dy}\int_{p(k)}^{y}{t^2(y-2t)F_2(T(t))dt}\\
=&\frac{(-1)^{k}}{\beta_{k-1}(x)^2}\int_{0}^{p(k)}{t^2F_2(T(t))dt}+\frac{(-1)^k}{\beta_{k-1}(x)^2}y^2(y-2y)F_2(T(y))\\
=&\frac{(-1)^{k}}{\beta_{k-1}(x)^2}\int_{0}^{p(k)}{t^2F_2(T(t))dt}+\frac{(-1)^{k+1}}{\beta_{k-1}(x)^2}T^{k}(x)^3F_2(T^{k+1}(x)),
\end{align*}
by the Fundamental Theorem of Calculus and the fact that $t^2(T^{k}(x)-2t)F_2(T(t))dt$ is continuous on $(p(k), T^k(x)]$.
\end{proof}

\begin{claim}\label{claim:deru1}
 Let $x \in (0,1)\setminus\Q$. For all $k\in\N$ we have
\begin{multline*}
 u_{3,k}'(x)=\beta_{k}(x)^2T^{k}(x)F_2(T^{k+1}(x))+\beta_{k-1}(x)^2 \int_{0}^{p(k)}{t^2F_2(T(t))dt}\\
+4(-1)^{k+1} \mathcal{I}_{k}(x)\beta_{k-1}(x)^4\sum_{j=0}^{k-1}(-1)^j\frac{T^j(x)}{\beta_{j}(x)^2}.
\end{multline*}
\end{claim}
\begin{proof}
 We have
\begin{align*}
  u_{3,k}'(x)
= & (-1)^{k+1} \mathcal{I}_{k}(x)(\beta_{k-1}(x)^4)'+(-1)^{k+1} (\mathcal{I}_{k}(x))'\beta_{k-1}(x)^4\\
= &4(-1)^{k+1} \mathcal{I}_{k}(x)\beta_{k-1}(x)^4\sum_{j=0}^{k-1}(-1)^j\frac{T^j(x)}{\beta_{j}(x)^2}\\
&+ (-1)^{k+1} \frac{(-1)^{k+1}}{\beta_{k-1}(x)^2}T^{k}(x)^3F_2(T^{k+1}(x))\beta_{k-1}(x)^4\\
&+(-1)^{k+1}\frac{(-1)^{k+1}}{\beta_{k-1}(x)^2}\int_{0}^{p(k)}{t^2F_2(T(t))dt}\beta_{k-1}(x)^4 &\textnormal{ by Claim~\ref{claim:intikx}}\\
= &\beta_{k}(x)^2T^{k}(x)F_2(T^{k+1}(x))+\beta_{k-1}(x)^2 \int_{0}^{p(k)}{t^2F_2(T(t))dt}\\
&+4(-1)^{k+1} \mathcal{I}_{k}(x)\beta_{k-1}(x)^4\sum_{j=0}^{k-1}(-1)^j\frac{T^j(x)}{\beta_{j}(x)^2}.
\end{align*}
This completes the proof of the claim.
\end{proof}

\begin{lem}\label{u1conv}
 The series
\begin{multline*}
 \sum_{k=0}^{\infty} \Big( \beta_{k}(x)^2T^{k}(x)F_2(T^{k+1}(x))+\beta_{k-1}(x)^2 \int_{0}^{p(k)}{t^2F_2(T(t))dt}\\
+4(-1)^{k+1} \mathcal{I}_{k}(x)\beta_{k-1}(x)^4\sum_{j=0}^{k-1}(-1)^j\frac{T^j(x)}{\beta_{j}(x)^2}\Big)
\end{multline*} 
converges for all $x \in (0,1)\setminus\Q$.
\end{lem}
\begin{proof}
Since $|F_2|$ is bounded by $\|F_2\|_\infty$, we have
 \begin{equation}\label{intzeta}
  |\mathcal{I}_k(x)|\leq  \|F_2\|_\infty\int_{0}^{T^{k}(x)}{|t^2(T^{k}(x)-2t)|dt}
\leq  \|F_2\|_\infty\int_{0}^{1}{|t^2||(T^{k}(x)-2t)|dt}
\leq \|F_2\|_\infty,
 \end{equation}
and
\begin{align*}
 \Big| \sum_{k=0}^{\infty}& \Big( \beta_{k}(x)^2T^{k}(x)F_2(T^{k+1}(x))+\beta_{k-1}(x)^2 \int_{0}^{p(k)}{t^2F_2(T(t))dt}\\
&+4(-1)^{k+1} \mathcal{I}_{k}(x)\beta_{k-1}(x)^4\sum_{j=0}^{k-1}(-1)^j\frac{T^j(x)}{\beta_{j}(x)^2}\Big)\Big|\\
\leq & \|F_2\|_\infty\sum_{k=0}^{\infty}\beta_{k}(x)^2+\|F_2\|_\infty\sum_{k=0}^{\infty}\beta_{k-1}(x)^2 \int_{0}^{p(k)}{t^2dt}\\
&+4\|F_2\|_\infty\sum_{k=0}^{\infty}\Big|\beta_{k-1}(x)^4\sum_{j=0}^{k-1}(-1)^j\frac{T^j(x)}{\beta_{j}(x)^2}\Big|\\
\leq & \|F_2\|_\infty\sum_{k=0}^{\infty}\beta_{k}(x)^2+\|F_2\|_\infty\sum_{k=0}^{\infty}\beta_{k-1}(x)^2 +4\|F_2\|_\infty\sum_{k=0}^{\infty}\beta_{k-1}(x)^3q_{k-1}\textnormal{ by Claim~\ref{caqkitfs}}\\
\leq & 6\|F_2\|_\infty\sum_{k=0}^{\infty}\frac{1}{q_k^2},
\end{align*}
by Proposition~\ref{fabag}~(1). It converges for all $x\in(0,1)\setminus\Q$ by Proposition~\ref{faqk}~(2).
\end{proof}

We can now prove Lemma~\ref{u1diff}.

\medskip

\noindent\textit{Proof of Lemma~\ref{u1diff}.} Let $x\in(0,1)\setminus\Q$. 
By the Mean Value Theorem and the fact that $u_{3,k}$ is continuous and differentiable on $I_k(x)$ for all $k \leq K_h$,
 we have $\frac{u_{3,k}(x+h)-u_{3,k}(x)}{h}=u_{3,k}'(t_k)$ for some $t_k$ between $x$ and $x+h$ for all $k \leq K_h$. Then $ \frac{\sum_{k=0}^{K_h-1}\left(u_{3,k}(x+h)-u_{3,k}(x)\right)}{h}=\sum_{k=0}^{K_h-1}u_{3,k}'(t_k).$
We have
\begin{multline*}
\Big|\frac{\sum_{k=0}^{K_h-1}\left(u_{3,k}(x+h)-u_{3,k}(x)\right)}{h}-\sum_{k=0}^{\infty} u_{3,k}'(x)\Big| = \Big|\sum_{k=0}^{K_h-1}u_{3,k}'(t_k)-\sum_{k=0}^{\infty} u_{3,k}'(x)\Big| \\
\leq \Big|\sum_{k=0}^{K_h-1}(u_{3,k}'(t_k)- u_{3,k}'(x))\Big|+\Big|\sum_{k=K_h}^{\infty} u_{3,k}'(x)\Big|.
\end{multline*}
By Lemma~\ref{u1conv}, $|\sum_{k=K_h}^{\infty} u_{3,k}'(x)|$ converges to 0 as $h \to 0$.
Then
\begin{multline}\label{eq:conu3}
 \Big|\sum_{k=0}^{K_h-1}(u_{3,k}'(t_k)- u_{3,k}'(x))\Big|\\
\leq \Big|\sum_{k=0}^{K_h-1}\Big(\beta_{k}(t_k)^2T^{k}(t_k)F_2(T^{k+1}(t_k))- \beta_{k}(x)^2T^{k}(x)F_2(T^{k+1}(x))\Big)\Big|\\
+\Big|\sum_{k=0}^{K_h-1}\Big(\beta_{k-1}(t_k)^2 \int_{0}^{p(k, t_k)}{t^2F_2(T(t))dt}-\beta_{k-1}(x)^2 \int_{0}^{p(k,x)}{t^2F_2(T(t))dt}\Big)\Big|\\
+\Big|\sum_{k=0}^{K_h-1}4(-1)^{k+1} \Big(\mathcal{I}_{k}(t_k)\beta_{k-1}(t_k)^4\sum_{j=0}^{k-1}(-1)^j\frac{T^j(t_k)}{\beta_{j}(t_k)^2}- \mathcal{I}_{k}(x)\beta_{k-1}(x)^4\sum_{j=0}^{k-1}(-1)^j\frac{T^j(x)}{\beta_{j}(x)^2}\Big)\Big|.
\end{multline}
We will now show that each of these terms converges to $0$ as $h\to0$.

\medskip

We start with the first term.
We have
\begin{multline*}
 \Big|\sum_{k=0}^{K_h-1}(\beta_{k}(t_k)^2T^{k}(t_k)F_2(T^{k+1}(t_k))- \beta_{k}(x)^2T^{k}(x)F_2(T^{k+1}(x)))\Big|\\
\leq \sum_{k=0}^{K_h-1}|\beta_{k}(t_k)^2T^{k}(t_k)- \beta_{k}(x)^2T^{k}(x)||F_2(T^{k+1}(t_k)|\\
+ \sum_{k=0}^{K_h-1}|F_2(T^{k+1}(t_k))-F_2(T^{k+1}(x))|\beta_{k}(x)^2T^{k}(x).
\end{multline*}
The function $\beta_{k}(y)^2T^{k}(y)$ is continuous and differentiable on $I_{K_h}(x)$ for all $k \leq K_h$, and by Proposition~\ref{fabi}~(2) we have
 $(\beta_{k}(y)^2T^{k}(y))'=2\beta_{k}(y)T^{k}(y)(-1)^kq_k+\beta_{k}(y)^2(-1)^k\frac{1}{\beta_k(y)^2}$. Thus for all $y \in I_{K_h}(x)$ we have
$|(\beta_{k}(y)^2T^{k}(y))'|\leq 3$. By the fact that $|F_2|$ is bounded by $\|F_2\|_\infty$, the Mean Value Theorem and Lemma~\ref{hitoqkh} we have
\begin{equation*}
 \sum_{k=0}^{K_h-1}|(\beta_{k}(t_k)^2T^{k}(t_k)- \beta_{k}(x)^2T^{k}(x)||F_2(T^{k+1}(t_k)| \leq \|F_2\|_\infty \sum_{k=0}^{K_h-1}3|h| \leq \frac{3\|F_2\|_\infty}{q_{K_h}},
\end{equation*}
which converges to 0 as $h \to 0$.
Let $N = q_{K_h}^2$. Using the same arguments as in the proof of Lemma~\ref{ftildeh}, for some constants $c_1, c_2$ we have
\begin{align*}
 \sum_{k=0}^{K_h-1}|F_2(T^{k+1}(t_k))&-F_2(T^{k+1}(x))|\beta_{k}(x)^2T^{k}(x)\\
\leq &c_1\sum_{k=0}^{K_h-1}|h|\beta_{k}(x)^2T^{k}(x)q_{k+1}^2 \log N+c_2 \sum_{k=0}^{K_h-1}\beta_{k}(x)^2T^{k}(x) \frac{1}{N^{3/4}}\\
\leq &c_1\frac{2}{q_{K_h}} \sum_{k=0}^{K_h-1}\frac{2\log q_{K_h}}{q_{K_h}}+c_2 \sum_{k=0}^{K_h-1}\frac{1}{q_k^2q_{K_h}^{3/2}},
\end{align*}
by Proposition~\ref{fabag}~(1) and Lemma~\ref{hitoqkh}, which converges to $0$ as $h \to 0$.

\medskip

For the second term, note that since for all $k \leq K_h$ we have $t_k \in I_{K_h}(x)$, then for $k \leq K_h$ we have that $p(k,t_k)=p(k,x)$. 
We will denote it $p(k)$. Since $\int_{0}^{p(k)}{t^2F_2(T(t))dt}$ is bounded by $\|F_2\|_\infty$ for all $k$, we have
\begin{multline*}
 \sum_{k=0}^{K_h-1}\left|\beta_{k-1}(t_k)^2 \int_{0}^{p(k)}{t^2F_2(T(t))dt}-\beta_{k-1}(x)^2 \int_{0}^{p(k)}{t^2F_2(T(t))dt}\right|\\
\leq \|F_2\|_\infty\sum_{k=0}^{K_h-1}|\beta_{k-1}(t_k)^2 -\beta_{k-1}(x)^2|\leq \|F_2\|_\infty\sum_{k=0}^{K_h-1} 2\frac{q_{k-1}}{q_k}|h| \leq  \frac{2\|F_2\|_\infty}{q_{K_h}}\sum_{k=0}^{K_h-1}\frac{1}{q_k}.
\end{multline*}
The last line follows from the fact that for all $k \leq K_h$ the function $\beta_{k-1}(y)^2$ is continuous and differentiable on $I_{K_h}(x)$; by Proposition~\ref{fabi}~(2.b) $(\beta_{k-1}(y)^2)'=2\beta_{k-1}(y)(-1)^{k-1}q_{k-1}$. 
Then by Proposition~\ref{fabag}~(1)
 $|(\beta_{k-1}(y)^2)'| \leq 2\frac{q_{k-1}}{q_k}$, for all $y \in I_k(x)$. By the Mean Value Theorem, the fact that $|h| \geq |x-t_k|$ and Lemma~\ref{hitoqkh} we obtain the result. 
It follows from Proposition~\ref{faqk}~(2) that the term converges to 0 as $h \to 0$.

\medskip

We consider the last term. Applying Claim~\ref{caqkitfs}, we get
\begin{multline*}
 \left|\sum_{k=0}^{K_h-1}4(-1)^{k+1} \left(\mathcal{I}_{k}(t_k)\beta_{k-1}(t_k)^4\sum_{j=0}^{k-1}(-1)^j\frac{T^j(t_k)}{\beta_{j}(t_k)^2}-\mathcal{I}_{k}(x)\beta_{k-1}(x)^4\sum_{j=0}^{k-1}(-1)^j\frac{T^j(x)}{\beta_{j}(x)^2}\right)\right|\\
\leq 4\sum_{k=0}^{K_h-1}q_{k-1}|\mathcal{I}_{k}(t_k)-\mathcal{I}_{k}(x)|\beta_{k-1}(t_k)^3+4\sum_{k=0}^{K_h-1}q_{k-1}|\beta_{k-1}(t_k)^3-\beta_{k-1}(x)^3||\mathcal{I}_{k}(x)|.
\end{multline*}
By Proposition~\ref{fabag}~(1) and bounding $|F_2|$ by $\|F_2\|_\infty$, we have
\begin{multline*}
 4\sum_{k=0}^{K_h-1}q_{k-1}|\mathcal{I}_{k}(t_k)-\mathcal{I}_{k}(x)|\beta_{k-1}(t_k)^3\leq 4\sum_{k=0}^{K_h-1}\frac{q_{k-1}}{q_k^3}|\mathcal{I}_{k}(t_k)-\mathcal{I}_{k}(x)| \\
\leq 4\sum_{k=0}^{K_h-1}\frac{q_{k-1}}{q_k^3}\Big(\Big|\int_{T^k(x)}^{T^k(t_k)}-2t^3 F_2(T(t))dt\Big|+ \Big|\int_{0}^{T^k(x)}t^2(T^k(t_k)-T^k(x))F_2(T(t)) dt\Big|\\
+\Big|\int_{T^k(x)}^{T^k(t_k)}t^2T^k(t_k)F_2(T(t))dt\Big|\Big)\\
\leq 4\|F_2\|_\infty\sum_{k=0}^{K_h-1}\frac{q_{k-1}}{q_k^3}\Big(2\Big|\int_{T^k(x)}^{T^k(t_k)}t^3 dt\Big|+ \Big|\int_{0}^{T^k(x)}t^2(T^k(t_k)-T^k(x)) dt\Big|+\Big|\int_{T^k(x)}^{T^k(t_k)}t^2dt\Big|\Big)\\
\leq 4\|F_2\|_\infty\sum_{k=0}^{K_h-1}\frac{q_{k-1}}{q_k^3}\Big(\frac{|T^k(t_k)^4-T^k(x)^4|}{2}\\
+ \int_{0}^{T^k(x)}t^2|T^k(t_k)-T^k(x)| dt
+\frac{|T^k(t_k)^3-T^k(x)^3|}{3}\Big).
\end{multline*}
By Proposition~\ref{fabi}, the functions $T^k(y)^4$, $T^k(y)^3$ and $T^k(y)$ are continuous and differentiable on $I_{K_h}(x)$ for all $k \leq K_h$ 
with $(T^k(y)^4)'=4(-1)^k\frac{T^k(y)^3}{\beta_{k-1}(y)^2}$, $(T^k(y)^3)'=3(-1)^k\frac{T^k(y)^2}{\beta_{k-1}(y)^2}$ and $(T^k(y))'=\frac{(-1)^k}{\beta_{k-1}(y)^2}$. 
It follows that for $y \in I_{K_h}(x)$ we have
$|(T^k(y)^4)'|\leq16q_{k}^2$, $|(T^k(x)^3)'|=12q_{k}^2$ and $|(T^k(x))'|=4q_{k}^2$. By the Mean Value Theorem, the fact that $|t_k-x| \leq |h|$ and Lemma~\ref{hitoqkh} we get
\begin{multline*}
 4\sum_{k=0}^{K_h-1}q_{k-1}|\mathcal{I}_{k}(t_k)-\mathcal{I}_{k}(x)|\beta_{k-1}(t_k)^3
\leq 4\|F_2\|_\infty|h|\sum_{k=0}^{K_h-1}\frac{q_{k-1}}{q_k}\Big(8+ 4\int_{0}^{T^k(x)}t^2 dt
+4\Big)\\
\leq 48\|F_2\|_\infty|h|\sum_{k=0}^{K_h-1}\frac{q_{k-1}}{q_k}\leq \frac{96\|F_2\|_\infty}{q_{K_h}} \sum_{k=0}^{K_h-1}\frac{1}{q_k},
\end{multline*}
which converges to $0$ as $h \to 0$ by Proposition~\ref{faqk}~(2). 
Also, for all $k \leq K_h$ the function $\beta_{k-1}(y)^3$ is continuous and differentiable on $I_{K_h}(x)$ and $(\beta_{k-1}(y)^3)'=3(-1)^{k-1}\beta_{k-1}(y)^2q_{k-1}$. Hence,
 $|(\beta_{k-1}(y)^3)'| \leq 3\frac{q_{k-1}}{q_k^2}$ for all $y \in I_k(x)$. By (\ref{intzeta}) and Lemma~\ref{hitoqkh}, we have
\begin{multline*}
 4\sum_{k=0}^{K_h-1}q_{k-1}|\beta_{k-1}(t_k)^3-\beta_{k-1}(x)^3||\mathcal{I}_{k}(x)|\leq 12\|F_2\|_\infty \sum_{k=0}^{K_h-1}q_{k-1}|h|\frac{q_{k-1}}{q_k^2}\\
\leq \frac{12\|F_2\|_\infty}{q_{K_h}}\sum_{k=0}^{K_h-1}\frac{1}{q_k},
\end{multline*}
which converges to 0 as $h \to 0$ by Proposition~\ref{faqk}~(2).

This shows that (\ref{eq:conu3}) converges to 0 as $h\to0$ completing the proof of Lemma~\ref{u1diff}.
\hfill\qedsymbol

\bigskip

\noindent\textit{Proof of Theorem~\ref{2irrIm}~(i).} Let $x\in(0,1)\setminus\Q$ be a square-Brjuno number satisfying (\ref{xtilde}) or (\ref{xtildere}). 
By (\ref{edft}) and Lemmas \ref{ftildeh}, \ref{u3diff}, \ref{u1diff} and \ref{u2lim} we conclude that $F_2$ is differentiable at $x$ and
\begin{align*}
 F'_2(x)=&\lim_{h\to0}\frac{F_2(x+h)-F_2(x)}{h}\\
=&\frac{\pi^3}{3}\sum_{k=0}^{\infty}\beta_{k-1}(x)\gamma_{k}(x)+\frac{4\pi^3}{3}\sum_{k=0}^{\infty}\Big((-1)^{k}\beta_{k-1}(x)^2\beta_{k}(x)\gamma_k(x)\sum_{j=0}^{k-1}(-1)^j\frac{T^j(x)}{\beta_{j}(x)^2}\Big)\\
&-\frac{\pi^3}{3}\sum_{k=0}^{\infty}\beta_{k-1}(x)^2+\sum_{k=0}^{\infty}(P(T^k(x)))'\beta_{k-1}(x)^2\\
&+4\sum_{k=0}^{\infty}(-1)^{k}P(T^k(x))\beta_{k-1}(x)^4\sum_{j=0}^{k-1}(-1)^j\frac{T^j(x)}{\beta_{j}(x)^2}\\
&+6\sum_{k=0}^{\infty} \Big(\beta_{k}(x)^2T^{k}(x)F_2(T^{k+1}(x))+\beta_{k-1}(x)^2 \int_{0}^{p(k)}{t^2F_2(T(t))dt}\\
&+4(-1)^{k+1}\int_{0}^{T^{k}(x)}{t^2(T^{k}(x)-2t)F_2(T(t))dt}\cdot\beta_{k-1}(x)^4\sum_{j=0}^{k-1}(-1)^j\frac{T^j(x)}{\beta_{j}(x)^2}\Big),
\end{align*}
where $(P(T^k(x)))'$ is the derivative of the polynomial $P$ evaluated at $T^k(x)$ and $p(k)$ is the smaller endpoint of the interval $I_k(x)$, 
that is $p(k)=\frac{p_k}{q_k}$ if $k$ is even, and  $p(k)=\frac{p_k+p_{k-1}}{q_k+q_{k-1}}$ if $k$ is odd.

\medskip

Suppose now that $x\in(0,1)\setminus\Q$ is not square-Brjuno. We will show that there exists a sequence $h_n \to 0$ such that $\frac{F_2(x+h_n)-F_2(x)}{h_n} \to \infty$ as $n\to \infty$. 
For each $n \in \N$ odd choose $h_n>0$ such that if $x\in I(a_1,a_2,..., a_k, a_{n+1})$, 
then $x+h_n \in I(a_1,a_2,..., a_k, a_{n+1}+2)\setminus\Q$. We have $x+h_n \in I_k(x)$, but $x+h_n \notin I_{n+1}(x)$, 
and $h_n \to 0$ as $n \to \infty$, and $[x, x+h_n]$ contains the basic interval $I(a_1,a_2,..., a_k, a_{n+1}+1)$. We also note that if $t \in [x, x+h_n]$ then 
\begin{equation}\label{eq:div} 
q_{n+1} \leq q_{n+1}(t) \leq 3q_{n+1},
\end{equation} which implies that
\begin{equation}\label{iept}
  \frac{1}{18q_{n+1}^2} <\frac{1}{q_{n+1}(t)(q_{n+1}(t)+q_n)}= |I(a_1,a_2,..., a_k, a_{n+1}+1)|<h_n\leq \frac{1}{q_n q_{n+1}}.
\end{equation}
By Equation~(\ref{fcffeninf}), we have
\begin{multline*}
 \frac{F_2(x+h)-F_2(x)}{h_n}=\frac{\frac{\pi^3}{3}\sum_{k=0}^{\infty}\left(u_{1,k}(x+h_n)-u_{1,k}(x)\right)}{h_n}\\
+\frac{\sum_{k=0}^{\infty}\left(u_{2,k}(x+h_n)-u_{2,k}(x)\right)}{h_n}
+\frac{6\sum_{k=0}^{\infty}\left(u_{3,k}(x+h_n)-u_{3,k}(x)\right)}{h_n}.
\end{multline*}
We will now show that the last two terms converge to some finite limits as $n\to\infty$.

\medskip

Since $\sum_{k=0}^{\infty}u_{2,k}(y)$ converges absolutely for all $y$, we have 
$\sum_{k=0}^{\infty}(u_{2,k}(x+h_n)-u_{2,k}(x))=\sum_{k=0}^{n}(u_{2,k}(x+h_n)-u_{2,k}(x))+\sum_{k=n+1}^{\infty}(u_{2,k}(x+h_n)-u_{2,k}(x))$.
By the same arguments as in the proof of Lemma~\ref{u3diff}, we conclude that $\frac{\sum_{k=0}^{n}u_{2,k}(x+h_n)-u_{2,k}(x)}{h_n}$ converges to some finite limit as $n \to \infty$. 
By Proposition~\ref{fabag}~(1) and since $0\leq|P(y)|\leq \|P\|_\infty$ for all $y\in(0,1)$, we have
\begin{multline*}
 \left|\frac{\sum_{k=n+1}^{\infty}\left(u_{2,n}(x+h_k)-u_{2,k}(x)\right)}{h_n}\right|\\
\leq  \frac{\sum_{k=n+1}^{\infty}\left(|P(T^k(x+h_n))|\beta_{k-1}(x+h_n)^4+|P(T^k(x))|\beta_{k-1}(x)^4\right)}{h_n}\\
\leq  \frac{\|P\|_\infty}{h_n}\sum_{k=n+1}^{\infty}\left(\frac{1}{(q_k(x+h_n))^4}+\frac{1}{q_k^4}\right)
\leq  18\|P\|_\infty\sum_{k=n+1}^{\infty}\left(\frac{1}{(q_k(x+h_n))^2}+\frac{1}{q_k^2}\right)
\end{multline*}
by (\ref{eq:div}) and (\ref{iept}). It converges to 0 as $n\to\infty$ by Proposition~\ref{faqk}~(2).

\medskip

Since $\sum_{k=0}^{\infty}u_{3,k}(y)$ converges absolutely for all $y$, we have 
$\sum_{k=0}^{\infty}(u_{3,k}(x+h_n)-u_{3,k}(x))=\sum_{k=0}^{n}(u_{3,k}(x+h_n)-u_{3,k}(x))+\sum_{k=n+1}^{\infty}(u_{3,k}(x+h_n)-u_{3,k}(x))$.
By the same arguments as in Lemma~\ref{u1diff} we conclude that 
$\frac{\sum_{k=0}^{n}(u_{3,k}(x+h_n)-u_{3,k}(x))}{h_n}$ converges to some finite limit as $n \to \infty$. By Proposition~\ref{fabag}~(1) and since $|F_2|$ is bounded by $\|F_2\|_\infty$ we have
\begin{multline*}
\left|\frac{\sum_{k=n+1}^{\infty}(u_{3,k}(x+h_n)-u_{3,k}(x))}{h_n}\right|
\leq \frac{\|F_2\|_\infty}{h_n}\sum_{k=n+1}^{\infty}\Big(\frac{1}{(q_k(x+h_n))^4}+\frac{1}{q_k^4}\Big)\\
\leq 18\|F_2\|_\infty\sum_{k=n+1}^{\infty}\Big(\frac{1}{(q_k(x+h_n))^2}+\frac{1}{q_k^2}\Big),
\end{multline*}
by (\ref{eq:div}) and (\ref{iept}). It converges to 0 as $n\to\infty$ by Proposition~\ref{faqk}~(2).

Since $\sum_{k=0}^{\infty}u_{1,k}(y)$ converges absolutely for all $y$, we have 
$\sum_{k=0}^{\infty}(u_{1,k}(x+h_n)-u_{1,k}(x))=\sum_{n=0}^{n}(u_{1,k}(x+h_n)-u_{1,k}(x))+\sum_{k=n+1}^{\infty}(u_{1,k}(x+h_n)-u_{1,k}(x))$.
By Proposition~\ref{fabag} (1) and (2), we have
\begin{multline*}
 \left|\frac{\sum_{k=n+1}^{\infty}(u_{1,k}(x+h_n)-u_{1,k}(x))}{h_n}\right| \leq \frac{1}{h_n}\sum_{k=n+1}^{\infty}\Big(\frac{\log(2q_{k+1}(x+h_n))}{(q_k(x+h_n))^3 q_{k+1}(x+h_n)}+\frac{\log(2q_{k+1})}{q_k^3 q_{k+1}} \Big)\\
 \leq 36\sum_{k=n+1}^{\infty}\Big(\frac{1}{q_k(x+h_n)}+\frac{1}{q_k}\Big), 
\end{multline*}
by (\ref{eq:div}) and (\ref{iept}). It converges to 0 as $n\to\infty$ by Proposition~\ref{faqk}~(2).

\medskip

As in the proof of Lemma~\ref{u2lim}, we have 
\begin{multline*}
 \frac{\sum_{k=0}^{n}(u_{1,k}(x+h_n)-u_{1,k}(x))}{h_n}=
-\sum_{k=0}^{n}\beta_{k-1}(x)^2\log(T^{k}(x+h_n))\\
+\sum_{k=0}^{n}4(-1)^{k+1}\beta_{k-1}(x)^3\beta_{k}(x)\log(T^{k}(x+h_n))\sum_{j=0}^{k-1}(-1)^j\frac{T^j(x)}{\beta_{j}(x)^2}\\
+\sum_{k=0}^{n}(-1)^{k+1}\frac{\beta_{k-1}(x)^3\beta_k(x)(\log(T^{k}(x+h_n))-\log(T^{k}(x)))}{h_n}\\
+\sum_{k=0}^{n}(-1)^{k+1}A_k h_n\log(T^{k}(x+h_n))+\sum_{k=0}^{n}(-1)^{k+1}B_k h_n^2\log(T^{k}(x+h_n))\\
+\sum_{k=0}^{n}(-1)^{k+1}C_k h_n^3\log(T^{k}(x+h_n)),
\end{multline*}
where $A_k, B_k, C_k$ were defined in (\ref{defabc}).
By the same arguments as in the proof of Lemma~\ref{u2lim}, we conclude that 
$\sum_{k=0}^{n}(-1)^{k+1}A_k h_n\log(T^{k}(x+h_n))$ $+\sum_{k=0}^{n}(-1)^{k+1}B_k h_n^2\log(T^{k}(x+h_n))$ $+\sum_{k=0}^{n}(-1)^{k+1}C_k h_n^3\log(T^{k}(x+h_n))$ converges to 0 as $n\to\infty$, 
and that $\sum_{k=0}^{n}4(-1)^{k+1}$ $\beta_{k-1}(x)^3 \beta_{k}(x) \log(T^{k}(x+h_n)) \sum_{j=0}^{k-1}(-1)^j \frac{T^j(x)}{\beta_{j}(x)^2}$ 
and $\sum_{k=0}^{n}\frac{(-1)^{k+1}}{h_n}\beta_{k-1}(x)^3\beta_k(x)$ $(\log(T^{k}(x+h_n))-\log(T^{k}(x)))$ both converge to finite limits as $n\to\infty$. 
Finally, we have
$$ -\sum_{k=0}^{n}\beta_{k-1}(x)^2\log(T^{k}(x+h_n))
=\sum_{k=0}^{n}\beta_{k-1}(x)\gamma_k(x)+\sum_{k=0}^{n}\beta_{k-1}(x)^2\log\Big(\frac{T^{k}(x)}{T^{k}(x+h_n)}\Big).$$
Since $h_n>0$, we have $x<x+h_n$. If $k$ is odd then $\log(\frac{T^{k}(x)}{T^{k}(x+h_n)})>0$, and if $k$ is even then $\beta_{k-1}(x)^2\log(\frac{T^{k}(x)}{T^{k}(x+h_n)}) \geq -\log (4) \beta_{k-1}(x)^2$ by Proposition~\ref{faqk}~(4).
Thus, we have 
\begin{multline*}
\sum_{k=0}^{n}-\beta_{k-1}(x)^2\log(T^{k}(x+h_n)) \geq \sum_{k=0}^{n}\beta_{k-1}(x)\gamma_k(x)-\sum_{\substack{
k=0, \\
k\text{ odd}}}^n\beta_{k-1}(x)^2\log (4), \\
\geq \sum_{k=0}^{n}\beta_{k-1}(x)\gamma_k(x)-\log(4)\sum_{k=0}^\infty\beta_{k-1}(x)^2.
\end{multline*}
By Propositions \ref{fabag}~(1) and \ref{faqk}~(2) we have $|-\log(4)\sum_{k=0}^\infty\beta_{k-1}(x)^2|<\infty$. Since $x$ is not square-Brjuno $-\sum_{k=0}^{n}\beta_{k-1}(x)^2\log(T^{k}(x+h_n)) \to \infty$ as $n \to \infty$.

This shows that $\frac{F_2(x+h_n)-F_2(x)}{h_n} \to \infty$ as $n \to \infty$, and we conclude that $F_2$ is not differentiable at $x$. This completes the proof of Theorem~\ref{2irrIm}~(i).
\hfill\qedsymbol

\subsection{Proof of Theorem~\ref{2irrIm} (ii)}

Let $x\in\R\setminus\Q$. Since $G_2$ is $1$-periodic, we may assume $x\in(0,1)$. For brevity, let
\begin{align*}
 v_{1,k}(x)=& \beta_{k-1}(x)\beta_{k}(x)^2\gamma_{k}(x)\\
 v_{2,k}(x)=& Q(T^{k}(x))\beta_{k-1}(x)^4\\
 v_{3,k}(x)=& \beta_{k-1}(x)^4\int_{0}^{T^{k}(x)}{t^2(T^{k}(x)-2t)G_2(T(t))dt} 
\end{align*}
By Corollary~\ref{cor1}, with this notation, for all $n\in\N$, we have
\begin{equation*}
G_2(x)=G_2(T^{n}(x))\beta_{n-1}(x)^4+\pi^2\sum_{k=0}^{n}v_{1,k}(x)+\sum_{k=0}^{n} v_{2,k}(x) +6\sum_{k=0}^{n}v_{3,k}(x).
\end{equation*}
 For each $h$, let $K_h \in \N$ such that $x+h\in I_k(x)$ for all $k\leq K_h$ and $x+h \notin I_{K_h+1}(x)$. We then have
\begin{multline}\label{edftr}
 \frac{G_2(x+h)-G_2(x)}{h}
=\frac{\left(G_2(T^{K_h-1}(x+h))\beta_{K_h-2}(x+h)^4-G_2(T^{K_h-1}(x))\beta_{K_h-2}(x)^4\right)}{h}\\
+\frac{\pi^2\sum_{k=0}^{K_h-1}\left(v_{1,k}(x+h)-v_{1,k}(x)\right)}{h}
+\frac{\sum_{k=0}^{K_h-1}\left(v_{2,k}(x+h)-v_{2,k}(x)\right)}{h}\\
+\frac{6\sum_{k=0}^{K_h-1}\left(v_{3,k}(x+h)-v_{3,k}(x)\right)}{h}.
\end{multline}
We proceed as in the proof of Part~(i) of Theorem~\ref{2irrIm}. We consider each summand as $h\to 0$.
\begin{lem}\label{ftildehr}
Let $x \in (0,1)\setminus\Q$ such that it satisfies (\ref{xtilde}) or (\ref{xtildere}), then 
$$\frac{\left(G_2(T^{K_h-1}(x+h))\beta_{K_h-2}(x+h)^4-G_2(T^{K_h-1}(x))\beta_{K_h-2}(x)^4\right)}{h} \to 0,$$ as $h \to 0$.
\end{lem}
\begin{proof}
The proof is very similar to the proof of the Lemma~\ref{ftildeh}, and therefore omitted. 
\end{proof}

\begin{lem}\label{u2limr}
Let $x \in (0,1)\setminus\Q$, then 
\begin{multline*}
\frac{\sum_{k=0}^{K_h-1}\left(v_{1,k}(x+h)-v_{1,k}(x)\right)}{h}\to 2\sum_{k=0}^{\infty}(-1)^{k}\beta_{k}(x)\gamma_{k}(x)\\
+4\sum_{k=0}^{\infty}\Big(\beta_{k-1}(x)\beta_k(x)^2\gamma_{k}(x)\sum_{j=0}^{k-1}(-1)^j\frac{T^j(x)}{\beta_{j}(x)^2}\Big)
+\sum_{k=0}^{\infty}(-1)^{k+1}\beta_{k-1}(x)\beta_k(x)<\infty,
\end{multline*}
as $h\to 0$.
\end{lem}
\begin{proof}
 The proof is very similar to the proof of Lemma~\ref{u2lim}, and therefore omitted.
\end{proof}

\begin{lem}\label{u3diffr}
Let $x \in (0,1)\setminus\Q$, then 
\begin{multline*}
\frac{\sum_{k=0}^{K_h-1}\left(v_{2,k}(x+h)-v_{2,k}(x)\right)}{h} \\
\to \sum_{k=0}^{\infty}(-1)^{k}(Q(T^k(x)))'\beta_{k-1}(x)^2+4\sum_{k=0}^{\infty}\Big(Q(T^k(x))\beta_{k-1}(x)^4\sum_{j=0}^{k-1}(-1)^j\frac{T^j(x)}{\beta_{j}(x)^2}\Big)<\infty,
\end{multline*}
as $h\to 0$, where $(Q(T^k(x)))'$ is the derivative of the polynomial $Q$ evaluated at $T^k(x)$. 
\end{lem}
\begin{proof}
The proof is very similar to the proof of the Lemma~\ref{u3diff}, and therefore omitted. 
\end{proof}

\begin{lem}\label{u1diffr}
Let $x \in (0,1)\setminus\Q$, then 
 \begin{multline*}
 \frac{\sum_{k=0}^{K_h-1}\left(v_{3,k}(x+h)-v_{3,k}(x)\right)}{h}\\
\to \sum_{k=0}^{\infty}\Big((-1)^{k+1}\beta_{k}(x)^2T^{k}(x)G_2(T^{k+1}(x))
+(-1)^{k+1}\beta_{k-1}(x)^2\int_{0}^{p(k)}{t^2G_2(T(t))dt}\\
+4\int_{0}^{T^{k}(x)}{t^2(T^{k}(x)-2t)G_2(T(t))dt}\cdot\beta_{k-1}(x)^4\sum_{j=0}^{k-1}(-1)^j\frac{T^j(x)}{\beta_{j}(x)^2}\Big) <\infty,
 \end{multline*}
as $h \to 0$ where $p(k)$ is the smaller endpoint of the interval $I_k(x)$, that is $p(k)=\frac{p_k}{q_k}$ if $k$ is even, and  $p(k)=\frac{p_k+p_{k-1}}{q_k+q_{k-1}}$ if $k$ is odd.
\end{lem}
\begin{proof}
The proof is very similar to the proof of the Lemma~\ref{u1diff}, and therefore omitted. 
\end{proof}

\medskip

\noindent\textit{Proof of Theorem~\ref{2irrIm}~(ii).} Let $x\in(0,1)\setminus\Q$ satisfy (\ref{xtilde}) or (\ref{xtildere}). 
By (\ref{edftr}) and Lemmas \ref{ftildehr}-\ref{u1diffr} we conclude that $G_2$ is differentiable at $x$ and
\begin{align*}
 G'_2(x)=&\lim_{h\to0}\frac{G_2(x+h)-G_2(x)}{h}\\
=&2\pi^2\sum_{k=0}^{\infty}(-1)^{k}\beta_{k}(x)\gamma_{k}(x)
+4\pi^2\sum_{k=0}^{\infty}\Big(\beta_{k-1}(x)\beta_k(x)^2\gamma_{k}(x)\sum_{j=0}^{k-1}(-1)^j\frac{T^j(x)}{\beta_{j}(x)^2}\Big)\\
&+\pi^2\sum_{k=0}^{\infty}(-1)^{k+1}\beta_{k-1}(x)\beta_k(x)+\sum_{k=0}^{\infty}(-1)^{k}(Q(T^k(x)))'\beta_{k-1}(x)^2\\
&+4\sum_{k=0}^{\infty}\Big(Q(T^k(x))\beta_{k-1}(x)^4\sum_{j=0}^{k-1}(-1)^j\frac{T^j(x)}{\beta_{j}(x)^2}\Big)\\
&+6\sum_{k=0}^{\infty}\Big((-1)^{k+1}\beta_{k}(x)^2T^{k}(x)G_2(T^{k+1}(x))
+(-1)^{k+1}\beta_{k-1}(x)^2\int_{0}^{p(k)}{t^2G_2(T(t))dt}\\
&+4\int_{0}^{T^{k}(x)}{t^2(T^{k}(x)-2t)G_2(T(t))dt}\cdot\beta_{k-1}(x)^4\sum_{j=0}^{k-1}(-1)^j\frac{T^j(x)}{\beta_{j}(x)^2}\Big)
\end{align*}
where $(Q(T^k(x)))'$ is the derivative of the polynomial $Q$ evaluated at $T^k(x)$ and $p(k)$ is the smaller endpoint of the interval $I_k(x)$, 
that is $p(k)=\frac{p_k}{q_k}$ if $k$ is even, and  $p(k)=\frac{p_k+p_{k-1}}{q_k+q_{k-1}}$ if $k$ is odd. \hfill\qedsymbol

\section{Proof of Theorem~\ref{moc}}\label{smoc}

\noindent\textit{Proof of Theorem~\ref{moc}.} Let $x\in(0,1)\setminus\Q$, let $y\in (0,1)$ and $K_y \in \N$ such that $y \in I_{K_y}(x)$, and $y \notin I_{K_y+1}(x)$. By Corollary~\ref{cor1}, we have

\begin{align*}
|F_2(x)&-F_2(y)|\leq|F_2(T^{K_y-1}(x))\beta_{K_y-2}(x)^4-F_2(T^{K_y-1}(y))\beta_{K_y-2}(y)^4|\\
&+\frac{\pi^3}{3}\sum_{k=0}^{K_y-1}|u_{1,k}(x)-u_{1,k}(y)|
+\sum_{k=0}^{K_y-1}|u_{2,k}(x)-u_{2,k}(y)|
+6\sum_{k=0}^{K_y-1}|u_{3,k}(x)-u_{3,k}(y)|,
\end{align*}
where $u_{1,k}, u_{2,k}, u_{3,k}$ were defined in (\ref{u1u2u3}).
We will consider each term separately. 

\medskip

Let $N=\lceil \frac{1}{|x-y|^2} \rceil$. By the same arguments as in Lemma~\ref{ftildeh}, we have
\begin{multline}\label{eq:moc1}
|F_2(T^{K_y-1}(x))\beta_{K_y-2}(x)^4-F_2(T^{K_y-1}(y))\beta_{K_y-2}(y)^4|\\
\leq
c_1|x-y| q_{K_y-1}^{-2} \log N
+c_2\frac{1}{N^{3/4}}\beta_{K_y-2}(y)^4
+4\|F_2\|_\infty|x-y|\\
\leq
2c_1|x-y| \log \Big(\frac{1}{|x-y|}\Big)
+(4\|F_2\|_\infty+c_2)|x-y|,
\end{multline}
for some constants $c_1, c_2$ independent of $x$ and $y$. 

We observe that $u_{1,k}, u_{2,k}, u_{3,k}$ are continuous and differentiable on $I_k(x)$ for all $k \leq K_y$. Therefore, by the Mean Value Theorem, for each $k$ there exists $t_k$ between $x$ and $y$ such that
\begin{multline*}
  \sum_{k=0}^{K_y-1}|u_{1,k}(x)-u_{1,k}(y)|\\
=\sum_{k=0}^{K_y-1}|x-y||(4\beta_{k-1}(t_k)^2\beta_{k}(t_k)q_{k-1}-\beta_{k-1}(t_k)^2)\log(T^{k}(t_k))-\beta_{k-1}(t_k)^2|\\
\leq |x-y|\Big(4(\log(2)+1)\sum_{k=0}^{\infty}\frac{1}{\text{Fib}_{k+1}}+\sum_{k=0}^{K_y-1}|\beta_{k-1}(t_k)^2\log(T^{k}(t_k))|+\sum_{k=0}^{\infty}\frac{1}{\text{Fib}_{k+1}^2}\Big).
\end{multline*}
Observe that for all $k \leq K_y-1$, we have $\frac{q_k}{2q_{k+1}} \leq T^{k}(t_k) \leq \frac{2q_k}{q_{k+1}}$, and hence
$\frac{1}{4} T^{k}(x) \leq T^{k}(t_k) \leq 4 T^{k}(x)$. We then have
\begin{multline*}
 \sum_{k=0}^{K_y-1}|\beta_{k-1}(t_k)^2\log(T^{k}(t_k))| \leq \sum_{k=0}^{K_y-1}\frac{1}{q_k^2}\log\left(\frac{1}{T^{k}(t_k)}\right) \leq \sum_{k=0}^{K_y-1}\frac{1}{q_k^2}\log\left(\frac{4}{T^{k}(x)}\right) \\
\leq \log(4)\sum_{k=0}^{\infty}\frac{1}{\text{Fib}_{k+1}^2}+\sum_{k=0}^{K_y-1}\frac{1}{q_k^2}\log\left(\frac{2q_{k+1}}{q_{k}}\right)
\leq \log(8)\sum_{k=0}^{\infty}\frac{1}{\text{Fib}_{k+1}^2}+\log(q_{K_y})\sum_{k=0}^{\infty}\frac{1}{\text{Fib}_{k+1}^2}.
\end{multline*}
We note that $y \in I_{K_y}(x)$ implies that $|x-y|\leq |I_{K_y}(x)| \leq \frac{1}{q_{K_y}^2}$, and hence 
$2\log(q_{K_y})\leq \log\left(\frac{1}{|x-y|}\right)$. We have
\begin{equation}\label{eq:moc2}
\sum_{k=0}^{K_y-1}|u_{1,k}(x)-u_{1,k}(y)|\leq c_3 |x-y|\log \Big(\frac{1}{|x-y|}\Big)+c_4|x-y|,
\end{equation}
with $c_3=\frac{1}{2}\sum_{k=0}^{\infty}\frac{1}{\text{Fib}_{k+1}^2}$ and $c_4=(\log(16)+5)\sum_{k=0}^{\infty}\frac{1}{\text{Fib}_{k+1}^2}$.

\medskip

By the Mean Value Theorem and the same arguments as in the proof of Lemma~\ref{u3diff}, for some $t_k$ between $x$ and $y$ we have
\begin{multline}\label{eq:moc3}
  \sum_{k=0}^{K_y-1}|u_{2,k}(x)-u_{2,k}(y)|=\sum_{k=0}^{K_y-1}|x-y||(P(T^k(t_k)))'\beta_{k-1}(t_k)^2-4P(T^k(t_k))\beta_{k-1}(t_k)^3q_{k-1}|\\
\leq |x-y|\Big(\sum_{k=0}^{K_y-1}\|P'\|_\infty\frac{1}{q_{k}^2}+4\sum_{k=0}^{K_y-1}\|P\|_\infty\frac{1}{q_{k}^2}\Big)\leq (\|P'\|_\infty+4\|P\|_\infty)|x-y| \sum_{k=0}^{\infty}\frac{1}{\text{Fib}_{k+1}^2},
\end{multline}
since $q_k(x)=q_k(t_k)$ for all $k \leq K_y$. 

\medskip

By the Mean Value Theorem and the same arguments as in the proof of Lemma~\ref{u1diff}, for some $t_k$ between $x$ and $y$ we have
\begin{multline}\label{eq:moc4}
 \sum_{k=0}^{K_y-1}|u_{3,k}(x)-u_{3,k}(y)|=\sum_{k=0}^{K_y-1}|x-y||\beta_{k}(t_k)^2T^{k}(t_k)F_2(T^{k+1}(t_k))|\\
+\beta_{k-1}(t_k)^2 \int_{0}^{p(k)}{t^2}F_2(T(t))dt+4(-1)^{k+1} \mathcal{I}_{k}(t_k)\beta_{k-1}(t_k)^4\sum_{j=0}^{k-1}(-1)^j\frac{T^j(t_k)}{\beta_{j}(t_k)^2}\bigg|\\
\leq |x-y|\left(\sum_{k=0}^{K_y-1}\frac{\|F_2\|_\infty}{q_{k+1}^2}+\sum_{k=0}^{K_y-1}\frac{\|F_2\|_\infty}{q_{k}^2}+4\sum_{k=0}^{K_y-1}\frac{\|F_2\|_\infty}{q_{k}^2}\right)\\
\leq 6\|F_2\|_\infty|x-y|\sum_{k=0}^{\infty}\frac{1}{\text{Fib}_{k+1}^2}.
\end{multline}
since $q_k(x)=q_k(t_k)$ for all $k \leq K_y$.

The result follows from (\ref{eq:moc1})-(\ref{eq:moc4}) with $C_1=2c_1+\frac{\pi^3}{6}\sum_{k=0}^{\infty}\frac{1}{\text{Fib}_{k+1}^2}$ and 
$C_2=4\|F_{2}\|_\infty   +c_2+(\frac{\pi^3}{3}(\log(16)+5)+\|P'\|_\infty+4\|P\|_\infty+36\|F_{2}\|_\infty)\sum_{k=0}^{\infty}\frac{1}{\text{Fib}_{k+1}^2}$.
\hfill\qedsymbol

\medskip

As we can see, we can choose constants $C_1$, $C_2$ independent of $x$.

\section{Case $k\geq 4$}\label{sgen}

To prove Conjecture~\ref{gencase}, we would proceed as in the case $k=2$. We would find a functional equation for 
$$\varphi_k(x)=G_k(x)+iF_k(x)$$ and then iterate it.

\subsection{Functional equation for $\varphi_k$}

In order to find the functional equation for $\varphi_k$ we use the connection to Eisenstein series. Recall that for $k \geq 4$ even the Eisenstein series $E_k$ is modular and it satisfies 
\begin{equation}\label{funeqek}
 E_k(t)=\frac{1}{t^k}E_k\left(-\frac{1}{t}\right),
\end{equation}
for all $t\in\uhp$, for details see for example \cite[III \S 2]{Kob}. 

\begin{thm}\label{funeqk}
 For $k \geq 4$ even, for $\alpha \in \uhp$, and $\tau \in \uhp$, we have
\begin{equation*}
\varphi_k(\tau)=\tau^{k+2}\varphi_k\left(-\frac{1}{\tau}\right)-\frac{k}{C_k}\tau \Log (\tau)+P_{k,\alpha}(\tau)
+\int_{\alpha}^{\tau}Q_{k,\alpha}(t,\tau)\varphi_k\left(-\frac{1}{t}\right)dt,
\end{equation*}
where $\Log$ denotes the principal value of the complex logarithm, $P_{k,\alpha}(\tau)$ is a polynomial 
in $\tau$ of degree less than or equal to $k+1$ depending on $\alpha$,
$Q_{k,\alpha}(t,\tau)$ is a polynomial in $t$ and $\tau$ of degree less than or equal to $k+1$ also depending on $\alpha$,
and $C_k=- \frac{k! 2k}{(2i\pi)^{k+1}B_k}$.
\end{thm}

\subsection{Proof of Theorem~\ref{funeqk}}

Throughout this section, let $4\leq k \in \N$ even, $\alpha \in \uhp$ and $\tau \in \uhp$ be all fixed. We make the following observations.
\begin{claim}\label{claim11}
We have
\begin{equation}\label{eq:1}
 C_k \cdot \varphi_k(\tau)=\int_{i\infty}^{\tau}(\tau-t)^k(E_k(t)-1)dt,
\end{equation}
 where $C_k = - \frac{k! 2k}{(2i\pi)^{k+1}B_k}$.
\end{claim}
\begin{proof}
 It follows by integrating the right-hand side of Equation~(\ref{eq:1}) by parts $k$ times.
\end{proof}

\begin{claim}
For $1 \leq j \leq k+1$ we have
\begin{equation}\label{eq:2}
 \varphi_k^{(j)}(\tau)=(2\pi i)^j \sum_{n=1}^{\infty}\frac{\sigma_{k-1}(n)}{n^{k+1-j}}e^{2i \pi n\tau},
\end{equation}
in particular
\begin{equation}\label{eq:3}
 \varphi_k^{(k+1)}(\tau)=\frac{k!}{C_k}(E_k(\tau)-1),
\end{equation}
 where $C_k$ is as in the Claim~\ref{claim11}.
\end{claim}
\begin{proof}
 We obtain (\ref{eq:2}) by differentiating $\varphi_k(\tau)$ $j$ times. Equality~(\ref{eq:3}) follows from (\ref{eq:2}) and the definition of Eisenstein series.
\end{proof}

\begin{claim}
We have
\begin{equation*}
 C_k \cdot \varphi_k(\tau)=\int_{\alpha}^{\tau}(\tau-t)^k E_k(t)dt+p_{k,\alpha}(\tau),
\end{equation*}
where $p_{k,\alpha}(\tau)$ is a polynomial in $\tau$ of degree less than or equal to $k+1$, which depends on $\alpha$.
In particular, $p_{k,\alpha}(\tau)=\frac{(\tau-\alpha)^{k+1}}{k+1}-\sum_{m=0}^{k} \frac{k! (\tau-\alpha)^{k-m}}{(k-m)!(2i\pi)^{k+1}}\varphi_k^{(k-m)}(\alpha)$.
\end{claim}
\begin{proof}
We note that Claim~\ref{claim11} implies that
\begin{equation}\label{eq:4}
 C_k \varphi_k(\tau)=\int_{\alpha}^{\tau}(\tau-t)^k E_k(t)dt-\int_{\alpha}^{\tau}(\tau-t)^k dt+\int_{i\infty}^{\alpha}(\tau-t)^k(E_k(t)-1)dt.
\end{equation}
We have
\begin{equation}\label{eq:5}
 -\int_{\alpha}^{\tau}(\tau-t)^k dt=\frac{(\tau-\alpha)^{k+1}}{k+1}.
\end{equation}
Then integrating by parts the last term in (\ref{eq:4}) $k$ times gives 
$$\int_{i\infty}^{\alpha}(\tau-t)^k(E_k(t)-1)dt=-\sum_{m=0}^{k} \frac{k! (\tau-\alpha)^{k-m}}{(k-m)!(2i\pi)^{k+1}}\varphi_k^{(k-m)}(\alpha),$$ where $\varphi_k^{(0)}$ denotes $\varphi_k$.
\end{proof}

Then by (\ref{funeqek}), we get
\begin{equation*}
 \int_{\alpha}^{\tau}(\tau-t)^k E_k(t)dt=\int_{\alpha}^{\tau}\frac{(\tau-t)^k}{t^k} E_k\left(-\frac{1}{t}\right)dt.
\end{equation*}

Substituting (\ref{eq:3}) with $\tau=-\frac{1}{t}$, we obtain 
\begin{multline}\label{eq:6}
\int_{\alpha}^{\tau}\frac{(\tau-t)^k}{t^k} E_k\left(-\frac{1}{t}\right)dt=
\int_{\alpha}^{\tau}\frac{(\tau-t)^k}{t^k} \left(1+\frac{C_k}{k!} \varphi_k^{(k+1)}\left(-\frac{1}{t}\right)\right)dt\\
=\int_{\alpha}^{\tau}\frac{(\tau-t)^k}{t^k}dt+ \frac{C_k}{k!}\int_{\alpha}^{\tau}\frac{(\tau-t)^k}{t^k} \varphi_k^{(k+1)}\left(-\frac{1}{t}\right)dt.
\end{multline}

\begin{claim}
 We have
\begin{equation}\label{eq:7}
 \int_{\alpha}^{\tau}\frac{(\tau-t)^k}{t^k}dt=-k\tau \Log (\tau)+q_{k,\alpha}(\tau),
\end{equation}
where $q_{k,\alpha}(\tau)$ is a polynomial in $\tau$ of degree less than or equal to $k$ depending of $\alpha$.
In particular, $q_{k,\alpha}(\tau)=\sum_{m=0}^{k-2}(-1)^m {k \choose m} \frac{1}{m-k+1}(\tau-\tau^{k-m}\alpha^{m-k+1})
+k\tau\Log(\alpha)+\tau-\alpha$.
\end{claim}
\begin{proof}
 To see that, we note 
$$\int_{\alpha}^{\tau}\frac{(\tau-t)^k}{t^k}dt=\int_{\alpha}^{\tau}\left(\sum_{m=0}^{k-2}(-1)^m {k \choose m} \tau^{k-m}t^{m-k}-k\tau t^{-1}+1\right)dt.$$
\end{proof}

\medskip

Substituting (\ref{eq:5}), (\ref{eq:6}) and (\ref{eq:7}) into (\ref{eq:4}) gives
\begin{equation}\label{eq:8}
\varphi_k(\tau)=-\frac{k}{C_k}\tau \Log (\tau)+ \frac{1}{k!}\int_{\alpha}^{\tau}\frac{(\tau-t)^k}{t^k}\varphi_k^{(k+1)}\left(-\frac{1}{t}\right)dt+\frac{1}{C_k}(p_{k,\alpha}(\tau)+q_{k,\alpha}(\tau)).
\end{equation}
It rests to evaluate the integral $\int_{\alpha}^{\tau}\frac{(\tau-t)^k}{t^k} \varphi_k^{(k+1)}\left(-\frac{1}{t}\right)dt$.

Then we have
\begin{claim}\label{claim22}
We have
\begin{equation*}
  \int_{\alpha}^{\tau}\frac{(\tau-t)^k}{t^k}\varphi_k^{(k+1)}\left(-\frac{1}{t}\right)dt=k! \tau^{k+2}\varphi_k\left(-\frac{1}{\tau}\right)+r_{k,\alpha}(\tau)
+ \int_{\alpha}^{\tau}s_{k,\alpha}(t,\tau)\varphi_k\left(-\frac{1}{t}\right)dt,
\end{equation*}
where $r_{k,\alpha}(\tau)$ is a polynomial in $\tau$ of degree less than or equal to $k+1$ depending on $\alpha$, and
$s_{k,\alpha}(t,\tau)$ is a polynomial in $t$ and $\tau$ of degree less than or equal to $k+1$.
\end{claim}
\begin{proof}
We use the substitution $u=-\frac{1}{t}$, and we have
\begin{equation*}
 \int_{\alpha}^{\tau}\frac{(\tau-t)^k}{t^k}\varphi_k^{(k+1)}\left(-\frac{1}{t}\right)dt=\int_{t=\alpha}^{\tau}u^{k-2}\left(\tau+\frac{1}{u}\right)^k \varphi_k^{(k+1)}(u)du.
\end{equation*}
For simplicity, we will define $v_0(x,\tau)=\left(-\frac{1}{x}\right)^{k-2}(\tau-x)^k$, and for $m \geq 0$, $v_{m}(x,\tau)=\left[\frac{\partial^m u^{k-2}\left(\tau+\frac{1}{u}\right)^k}{\partial u^m}\right]_{u=-\frac{1}{x}}$.
By Leibniz product formula for $m \leq k$, we have
\begin{align*}
v_{m}\left(-\frac{1}{u},\tau\right)&=\frac{\partial^m u^{k-2}\left(\tau+\frac{1}{u}\right)^k}{\partial u^m}\\
&=\sum_{j=0}^m(-1)^j\frac{m!k!(j+1)}{(m-j)!(k-m+j)!}u^{-2+k-m}\tau^{m-j}\left(\tau+\frac{1}{u}\right)^{k-m+j}.
\end{align*}
Hence for all $0 \leq m \leq k$, we have $v_m(\tau,\tau)=0$.
Then integrating by parts $k$ times gives
\begin{multline*}
\int_{t=\alpha}^{\tau}u^{k-2}\left(\tau+\frac{1}{u}\right)^k \varphi_k^{(k+1)}(u)du=\sum_{i=0}^{k-1}(-1)^{i+1} v_{i}(\alpha,\tau)\varphi_k^{(k-i)}\left(-\frac{1}{\alpha}\right)\\
+ \int_{t=\alpha}^{\tau}v_{k}\left(-\frac{1}{u},\tau\right)\varphi'_k(u)du.
\end{multline*}
We observe that $\sum_{i=0}^{k-1}(-1)^{i+1} v_{i}(\alpha,\tau)\varphi_k^{(k-i)}\left(-\frac{1}{\alpha}\right)$ is a polynomial in $\tau$ of degree less than or equal to $k$.
We then note that 
$$v_k\left(-\frac{1}{u},\tau\right)=k!u^{-2}\tau^{k}+\sum_{j=1}^k(-1)^j\frac{k!k!(j+1)}{(k-j)!(j)!}u^{-2}\tau^{k-j}\left(\tau+\frac{1}{u}\right)^{j}.$$
For simplicity, write $w_k\left(-\frac{1}{u},\tau\right)=\sum_{j=1}^k(-1)^j\frac{k!k!(j+1)}{(k-j)!(j)!}u^{-2}\tau^{k-j}\left(\tau+\frac{1}{u}\right)^{j}$. 
We then have:
\begin{enumerate}
 \item $w_k(\tau,\tau)=0$;
 \item $w_k(\alpha,\tau)$ is a polynomial in $\tau$ of degree less then or equal to $k$;
 \item $\frac{\partial w_k\left(-\frac{1}{u},\tau\right)}{\partial u}=-\sum_{j=1}^k(-1)^j\frac{k!k!(j+1)}{(k-j)!(j)!}\tau^{k-j}u^{-3}\left(2\left(\tau+\frac{1}{u}\right)^{j}+ju^{-1}\left(\tau+\frac{1}{u}\right)^{j-1}\right)$;
 \item $\left[\frac{\partial w_k\left(-\frac{1}{u},\tau\right)}{\partial u}\right]_{u=-\frac{1}{t}}$ can be written as $t^{2}w_{k+1,\alpha}(t,\tau)$, 
where $w_{k+1,\alpha}(t,\tau)$ is a polynomial in $t$ and $\tau$ of degree less than or equal to $k+1$.
\end{enumerate}

Therefore we have
\begin{multline*}
\int_{t=\alpha}^{\tau}v_{k}\left(-\frac{1}{u},\tau\right)\varphi'_k(u)du=k!\tau^{k+2}\varphi_k\left(-\frac{1}{\tau}\right)-k!\alpha^{k}\varphi_k\left(-\frac{1}{\alpha}\right)- w_k(\alpha,\tau)\varphi_k\left(-\frac{1}{\alpha}\right)\\
+\int_{\alpha}^{\tau}\left(s_{k,\alpha}(t,\tau)+2k!t\tau^k \right)\varphi_k\left(-\frac{1}{t}\right)dt.
\end{multline*}
Letting $r_{k,\alpha}(\tau)=\sum_{i=0}^{k-1}(-1)^{i+1} v_{i}(\alpha,\tau)\varphi_k^{(k-i)}\left(-\frac{1}{\alpha}\right)-k!\alpha^{k}\varphi_k\left(-\frac{1}{\alpha}\right)- w_k(\alpha,\tau)\varphi_k\left(-\frac{1}{\alpha}\right)$, 
and $s_{k,\alpha}(t,\tau)=w_{k+1,\alpha}(t,\tau)+2k!t\tau^k$, gives the result.
\end{proof}

\bigskip

\noindent \textit{Proof of Theorem~\ref{funeqk}.}
If follows from Equations (\ref{eq:7}) and (\ref{eq:8}) that for $\alpha \in \uhp$, and $\tau \in \uhp$, we have
\begin{equation*}
\varphi_k(\tau)=\tau^{k+2}\varphi_k\left(-\frac{1}{\tau}\right)-\frac{k}{C_k}\tau \Log (\tau)+P_{k,\alpha}(\tau)
+\int_{\alpha}^{\tau}Q_{k,\alpha}(t,\tau)\varphi_k\left(-\frac{1}{t}\right)dt,
\end{equation*}
where $P_{k,\alpha}(\tau)=\frac{1}{C_k}(p_{k,\alpha}(\tau)+q_{k,\alpha}(\tau))+\frac{1}{k!}r_{k,\alpha}(\tau)$ is a polynomial 
in $\tau$ of degree less than or equal to $k+1$, and
$Q_{k,\alpha}(t,\tau)= \frac{1}{k!}s_{k,\alpha}(t,\tau)$ is a polynomial in $t$ and $\tau$ of degree less than or equal to $k+1$. This completes the proof of Theorem~\ref{funeqk} \hfill \qedsymbol

\subsection{Heuristic approach to Conjecture~\ref{gencase}}

We assume we can let $\alpha \to 0$. For $x \in \R^+$, letting $\tau \to x$, we get:
\begin{equation}\label{fern}
\varphi_k(x)=x^{k+2}\varphi_k\left(-\frac{1}{x}\right)-\frac{k}{C_k}x \log (x)+P_{k,0}(x)
+\int_{0}^{x}Q_{k,0}(t,x)\varphi_k\left(-\frac{1}{t}\right)dt.
\end{equation}
We read the behaviour of $F_2$ and $G_2$ around 0 from this equation. In order to prove part (i) of Conjecture~\ref{gencase}, 
we would find another functional equation for $\varphi_k$ in a similar way to the proof of Theorem~\ref{funeqk}. 
We would apply the modular property of $E_k$, namely that for all $t\in\uhp$ and for all $\gamma=\bigl(\begin{smallmatrix}
a&b\\ c&d
\end{smallmatrix} \bigr) \in SL_2(\Z)$, we have
$E_k(t)=\frac{E_k(\gamma\cdot t)}{(ct+d)^k},$
in the calculations instead of (\ref{funeqek}).

Taking imaginary parts on both sides of Equation~(\ref{fern}), we get
\begin{equation*}
F_k(x)=x^{k+2}F_k\left(-\frac{1}{x}\right)+D_k x \log (x)+P_{k}(x)
+\int_{0}^{x}Q_{k}(t,x)F_k\left(-\frac{1}{t}\right)dt,
\end{equation*}
where $Q_{k}(t,x)=\im(Q_{k,0}(t,x))$, $P_{k}(x)=\im(P_{k,0}(x))$, and $D_k=\frac{(2i)^k\pi^{k+1}B_k}{k!}$.
Taking real parts on both sides of Equation~(\ref{fern}), we get
\begin{equation*}
G_k(x)=x^{k+2}G_k\left(-\frac{1}{x}\right)+R_{k}(x)+\int_{0}^{x}S_{k}(t,x)G_k\left(-\frac{1}{t}\right)dt,
\end{equation*}
where $S_{k}(t,x)=\re(Q_{k,0}(t,x))$ and $R_{k+1}(x)=\re(P_{k,0}(x))$.

\begin{claim}
 Let $x \in (0,1)\backslash \Q$. Assume that (\ref{fern}) holds.
\begin{enumerate}
 \item We have:
 \begin{align*}
  F_k(x)&=-x^{k+2}F_k(T(x))+D_k x \log (x)+P_{k}(x)-\int_{0}^{x}Q_{k}(t,x)F_k(T(t))dt;\\
  G_k(x)&= x^{k+2}G_k(T(x))+R_{k}(x)+\int_{0}^{x}S_{k}(t,x)G_k(T(t))dt.
 \end{align*}
 \item For all $n \in \N$ we have
\begin{align*}
 F_k(x)=&(-1)^{n+1}\beta_n(x)^{k+2}F_k(T^{n+1}(x))-D_k\sum_{j=0}^n(-1)^j\beta_{j-1}(x)^{k} \beta_{j}(x) \gamma_j(x)\\
&+\sum_{j=0}^n(-1)^j\beta_{j-1}(x)^{k+2}P_{k}(T^j(x))\\
&+\sum_{j=0}^n(-1)^{j+1}\beta_{j-1}(x)^{k+2}\int_{0}^{T^j(x)}Q_{k}(t,T^j(x))F_k(T(t))dt;\\
 G_k(x)=&\beta_n(x)^{k+2}G_k(T^{n+1}(x))+\sum_{j=0}^n\beta_{j-1}(x)^{k+2}R_{k}(T^j(x))\\
&+\sum_{j=0}^n\beta_{j-1}(x)^{k+2}\int_{0}^{T^j(x)}S_{k}(t,T^j(x))G_k(T(t))dt;
\end{align*}
\item letting $n\to \infty$ we obtain
\begin{align}
F_k(x)=&-D_k\sum_{j=0}^\infty(-1)^j\beta_{j-1}(x)^{k} \beta_{j}(x) \gamma_j(x)+\sum_{j=0}^\infty(-1)^j\beta_{j-1}(x)^{k+2}P_{k}(T^j(x))\notag\\
&+\sum_{j=0}^\infty(-1)^{j+1}\beta_{j-1}(x)^{k+2}\int_{0}^{T^j(x)}Q_{k}(t,T^j(x))F_k(T(t))dt;\notag\\
 G_k(x)=&\sum_{j=0}^\infty\beta_{j-1}(x)^{k+2}R_{k}(T^j(x))+\sum_{j=0}^\infty\beta_{j-1}(x)^{k+2}\int_{0}^{T^j(x)}S_{k}(t,T^j(x))G_k(T(t))dt.\label{eq:10}
\end{align}
\end{enumerate}
\end{claim}
\begin{proof}
 The proof is very similar to the proof of Proposition~\ref{funeq2} and therefore omitted.
\end{proof}

We then have
\begin{claim}\label{claim:idontknow}
 Let $x\in(0,1)\setminus\Q$. For all $j \in \N$ we have that $\beta_{j-1}(x)^{k} \beta_{j}(x) \gamma_j(x)$ is differentiable at $x$ and 
\begin{multline*}
 (\beta_{j-1}(x)^{k} \beta_{j}(x) \gamma_j(x))'=(-1)^j\beta_{j-1}(x)^{k+2}\\
+(-1)^{j}(k+2)\beta_{j-1}(x)^{k-1}\beta_{j}(x)\gamma_{j}(x)q_{j-1} -(-1)^j\beta_{j-1}(x)^{k-1}\gamma_j(x).
\end{multline*}

We also have 
\begin{equation*}
  \left|\sum_{j=0}^\infty\left(\beta_{j-1}(x)^{k+2}+(k+2)\beta_{j-1}(x)^{k-1}\beta_{j}(x)\gamma_{j}(x)q_{j-1}
-\beta_{j-1}(x)^{k-1}\gamma_j(x)\right)\right| < \infty,
\end{equation*}
if and only if 
\begin{equation*}
 \sum_{j=0}^\infty \frac{\log(q_{j+1})}{q_j^k}<\infty.
\end{equation*}
\end{claim}
\begin{proof}
 The proof is very similar to the proofs of Lemmas~\ref{u2lim} and \ref{gambeta} and therefore omitted.
\end{proof}

\begin{claim}
Let $x\in(0,1)\setminus\Q$. For all $j \in \N$ we have that $\beta_{j-1}(x)^{k+2}P_{k}(T^j(x))$ is differentiable at $x$ and 
\begin{multline*}
 (\beta_{j-1}(x)^{k+2}P_{k}(T^j(x)))'\\
=(-1)^{j-1}(k+2)\beta_{j-1}(x)^{k+1}q_{k-1}P_{k}(T^j(x))+(-1)^j\beta_{j-1}(x)^{k}P'_{k}(T^j(x)),
\end{multline*}
where $P'_{k}(T^j(x))$ is the derivative of $P_{k}(y)$ with respect to $y$ evaluated at $T^j(x)$.

We also have that 
\begin{equation*}
 \left|\sum_{j=0}^\infty\left(-(k+2)\beta_{j-1}(x)^{k+1}q_{j-1}P_{k}(T^j(x))
+\beta_{j-1}(x)^{k}P'_{k}(T^j(x))\right)\right| < \infty,
\end{equation*}
for all $x \in (0,1)\backslash \Q$.
\end{claim}
\begin{proof}
 The proof is very similar to the proofs of Lemmas~\ref{u3diff} and \ref{s3sum} and therefore omitted.
\end{proof}

\begin{claim}\label{claim:idontknow2}
 Let $x\in(0,1)\setminus\Q$. We have that $\beta_{j-1}(x)^{k+2}\int_{0}^{T^j(x)}Q_{k}(t,T^j(x))F_k(T(t))dt$ is differentiable at $x$ for all $j \in \N$ and 
\begin{multline*}
 (\beta_{j-1}(x)^{k+2}\int_{0}^{T^j(x)}Q_{k}(t,T^j(x))F_k(T(t))dt)'\\
=(-1)^{j-1}(k+2)\beta_{j-1}(x)^{k+1}q_{j-1}\int_{0}^{T^j(x)}Q_{k}(t,T^j(x))F_k(T(t))dt\\
+(-1)^j\beta_{j-1}(x)^{k}\int_{0}^{p(j)}Q'_{k}(t,T^j(x))F_k(T(t))dt\\
+(-1)^j\beta_{j-1}(x)^{k}Q_{k}(t,T^j(x))F_k(T^{j+1}(t)),
\end{multline*}
where  $Q'_{k}(t,T^j(x))$ is the derivative of $Q_{k}(t,y)$ with respect to $y$ evaluated at $y=T^j(x)$, 
and $p(j)$ is the smaller endpoint of the interval $I_{j}(x)$.

We also have
\begin{multline*}
 \Bigg|\sum_{j=0}^\infty\Bigg( (k+2)\beta_{j-1}(x)^{k+1}q_{j-1}\int_{0}^{T^j(x)}Q_{k}(t,T^j(x))F_k(T(t))dt\\
-\beta_{j-1}(x)^{k}\int_{0}^{p(j)}Q'_{k}(t,T^j(x))F_k(T(t))dt\\
-\beta_{j-1}(x)^{k}Q_{k}(t,T^j(x))F_k(T^{j+1}(t))\Bigg) \Bigg| <\infty,
\end{multline*}
for all $x \in (0,1)\backslash \Q$.
\end{claim}
\begin{proof}
 The proof is very similar to the proofs of Lemmas~\ref{u1diff} and \ref{u1conv} and therefore omitted.
\end{proof}

Supposing that we can let $\alpha\to 0$ in Theorem~\ref{funeqek}.
The individual terms in the two sums in (\ref{eq:10}) are differentiable at every $x\in(0,1)\setminus\Q$ and the sums of the derivatives evaluated at $x\in(0,1)\setminus\Q$ converge. 
Since we are dealing with infinite sums, we cannot say that the derivative of $F_{k}(x)$ is the sum of derivatives from Claims~\ref{claim:idontknow}-\ref{claim:idontknow2} over $j\in\N$. 
Formally, to prove the conjecture (ii) and (iii), we would proceed as in the case $k=2$ first showing that we can let $\alpha\to 0$ in Theorem~\ref{funeqek}.

\bibliographystyle{alpha}
\bibliography{MyBiblio}

\end{document}